\documentclass[12pt]{amsart}
\usepackage{times, pdfpages, graphicx, bbm}
\usepackage{comment}
\usepackage{exscale}
\usepackage{tikz-cd}
\usepackage{epsfig}

\usepackage{amsthm}
\usepackage{amsfonts}

\usepackage{graphics,psfrag,graphicx,subfigure,epsfig, xypic}
\usepackage{amsmath,amsfonts,amssymb,amstext,amsthm,amscd,mathrsfs}
\usepackage[english]{babel}
\usepackage[all]{xy}
\usepackage{multicol}
\usepackage{rotating}

\usepackage{amsmath, amsfonts, amssymb, amscd, enumerate}
\usepackage{color}
\usepackage[all]{xy}

\theoremstyle{plain}
\newtheorem{theo}{Theorem}[section]

\theoremstyle{definition}

\newtheorem{example}[theo]{Example}
\newtheorem{definition}[theo]{Definition}

\theoremstyle{plain}
\newtheorem{lemma}[theo]{Lemma}
\newtheorem{theorem}[theo]{Theorem}

\newtheorem{proposition}[theo]{Proposition}

\theoremstyle{definition}
\newtheorem{problem}[theo]{Question}
\newtheorem{remark}[theo]{Remark}

\newcommand{\beq}{\begin{equation}}
\newcommand{\eeq}{\end{equation}}

\newcommand{\diam}{\operatorname{diam}}

\newcommand{\ve}{\varepsilon}

\renewcommand{\k}{\kappa}

\newcommand{\C}{\mathbb{C}}

\newcommand{\R}{\mathbb{R}}
\renewcommand{\H}{\mathbb{H}}
\newcommand{\bH}{\mathbb{H}}
\newcommand{\Z}{\mathbb{Z}}

\newcommand{\bN}{\mathbb{N}}

\newcommand{\bS}{\mathbb{S}}

\newcommand{\cI}{\mathcal{I}}

\newcommand{\ra}{\rightarrow}

\renewcommand{\square}{\kern1pt\vbox
{\hrule height 0.6pt\hbox{\vrule width 0.6pt\hskip 3pt
\vbox{\vskip 6pt}\hskip 3pt\vrule width 0.6pt}\hrule height0.6pt}\kern1pt}

\renewcommand\Re{\operatorname{Re}}
\renewcommand\Im{\operatorname{Im}}

\newcommand{\grad}{{\mathrm{grad}\,}}

\renewcommand{\Re}{{\rm Re}}
\renewcommand{\Im}{{\rm Im}}

\newcommand\refl[1]{R(#1)}

\newcommand{\be}{\begin{equation}}
\newcommand{\ee}{\end{equation}}

\def\<#1,#2>{\langle\,#1,\,#2\,\rangle}
\newcommand{\arr}{\begin{array}{rlll}}
\newcommand{\ea}{\end{array}}
\newcommand{\bea}{\begin{eqnarray}}
\newcommand{\eea}{\end{eqnarray}}
\newcommand{\bean}{\begin{eqnarray*}}
\newcommand{\eean}{\end{eqnarray*}}

\newcommand{\vph}{\varphi}

\def\sideremark#1{\ifvmode\leavevmode\fi\vadjust{%            The remark
\vbox to0pt{\hbox to 0pt{\hskip\hsize\hskip1em%               will appear only
\vbox{\hsize3cm\tiny\raggedright\pretolerance10000%          on the side
\noindent #1\hfill}\hss}\vbox to8pt{\vfil}\vss}}}%           in 3cm
\newcounter{ssig}
\newcounter{ttig}
\begin{document}
%\title[Second Cousin problem for slice--regular functions]{Second Cousin problem for slice--regular functions}
\title[Cartan Coverings]{Quaternionic Cartan coverings and applications}
\
\author{Jasna Prezelj}
\address{UL FMF, Jadranska 19,
  Ljubljana, Slovenija, UP FAMNIT, Glagolja\v ska 8,  Koper, Slovenija, IMFM, Jadranska 19, Ljubljana, Slovenija}
\email { jasna.prezelj@fmf.uni-lj.si}
\author {Fabio  Vlacci}\address{MIGe Universit\`a di Trieste Piazzale Europa 1,\ Trieste,
  Italy} \email{ fvlacci@units.it} \thanks{\rm The first author
  was partially supported by research program P1-0291 and by research
  projects N1-0237, J1-3005 at Slovenian Research Agency.  The second author was partially supported by
  Progetto MIUR di Rilevante Interesse Nazionale PRIN 2010-11 {\it
    Variet\`a reali e complesse: geometria, topologia e analisi
    armonica}.  The research that led to the present paper was
  partially supported by a grant of the group GNSAGA of Istituto
  Nazionale di Alta Matematica `F: Severi'.}

\maketitle
\begin{abstract}
We present the topological foundations for solvability of  multiplicative Cousin problems formulated on an axially symmetric domain $\Omega \subset \H.$ In particular, we provide a geometric construction of quaternionic Cartan coverings, which are generalizations of (complex) Cartan coverings as presented in Section 4 of \cite{fp}. Because of the requirements of symmetry inherent to the domains of definition of quaternionic regular functions, the existence of quaternionic Cartan coverings of $\Omega$ is not a consequence of the existence of complex Cartan coverings; for the latter, there are no requirements  for the symmetries with respect to the real axis. Due to the real axis's special role, also the covering restricted to $\Omega \cap \R$ must have additional properties. All these required properties  were achieved by starting from a particular symmetric tiling of the symmetric set  $\Omega \cap (\R + i\R)$. Finally, we apply these results to prove the vanishing of 'antisymmetric' cohomology groups of planar symmetric domains for $n \geq 2$.
\end{abstract}

\section{Introduction}\label{intro}

 We denote by $\mathbb{H}$ the algebra of quaternions.
If $\Omega\subset \mathbb{H}$  is an axially symmetric open set, the set of slice--regular functions in $\Omega$ will be denoted by $\mathcal{SR}(\Omega)$.
The theory of slice--regular functions (shortly recalled in
Section \ref{prelim} where we mainly address our attention on specific features of
$\mathcal{SR}(\Omega)$ which will be useful later)
clearly shows many promising aspects to become a good framework to
consider the generalizations of Cousin problems for quaternionic
functions; indeed, the notion of slice--regularity is a good extension
of the notion of holomorphicity for quaternionic functions and
semi-regular functions play the role of meromorphic functions.  The
analogues of Weierstrass and of  Mittag-Leffler Theorems are
already obtained in the framework of slice--regular functions (see \cite{GSS}). Furthermore, domains of holomorphy are
generalized in terms of quaternionic axially symmetric domains.

In the present paper, we provide the topological foundation for the proofs of the analogues of the above theorems for an arbitrary axially symmetric domain $\Omega \subset \H.$
Glueing local solutions is the main technique for proving these theorems in the holomorphic setting; requested compatibility is related to Cousin problems. To deal with these problems in the framework of slice--regular functions,
 one has to work with Cartan coverings  (Definition \ref{CartanSequence}), which are  quaternionic analogues of Cartan coverings of domains in $\C$ with some additional properties. First, we require that no four distinct sets of a Cartan covering intersect, i.e. a Cartan covering has the order at most $3$; in addition, due to the fact
 that only the real numbers form the center of $\mathbb{H}$ - as opposed to complex numbers, where the product is commutative, we require, among other things, that the subcovering $\mathcal B_{\R} \subset \mathcal B$ of a Cartan covering $\mathcal B$ defined by  $\mathcal B_{\R} := \{U \in \mathcal B, U \cap \R  \ne \varnothing\}$ has the order $2$ and  forms `a chain', so no three distinct sets intersect. The construction of Cartan coverings on axially symmetric domains is the paper's core and occupies most of Section \ref{coverings}. Section \ref{prelim} gives some preliminaries on slice--regular functions and axially symmetric domains and Section \ref{group} presents the properties of the set of slice--regular functions on finite unions of disjoint basic sets.

The main result of this paper is the existence of such coverings for axially symmetric domains in $\mathbb{H}$.

\begin{theorem}[Main Theorem]\label{BasicCovering} Let ${\mathcal U}$ be a locally finite axially symmetric
  open covering of an axially symmetric domain $\Omega\subset \mathbb{H}$ and let $Z \subset
  \Omega$ be a discrete set of points or spheres. Then $\Omega $ admits a
  Cartan covering $\mathcal B$ %n axially symmetric Cartan covering
  subordinated to $({\mathcal U},Z).$

  Even more, there exists a Cartan covering $\mathcal B = \{B_n\}_{ n \in \bN_0}$ of $\Omega$
  and a sequence $\{\ve_n\}_{n \in \bN_0}$ such that also the coverings
  \[ {\mathcal B}^t:= \{ B_n + B(0,t\ve_n)\}_{ n \in \bN_0} \]
  are Cartan coverings of $\Omega$  subordinated to  $({\mathcal U},Z)$ for all $t \in [0,1].$
\end{theorem}
In the last section  (Section \ref{applications}), we apply Cartan coverings to prove a theorem, which is similar to results on vanishing of $H^2(D,\Z)$  for complex domains $D$ but  with additional symmetry properties.

 Initially, this paper was only a preliminary part of a longer paper about the existence of the solutions to Cousin problems in the framework of slice--regular functions; for the sake of the reader, the authors have decided to present the sections on Cartan coverings as a separate paper because of its potential interest also in different settings. At the same time, the results on the solutions of Cousin problems in  $\mathcal{SR}(\Omega)$, which will appear in a forthcoming paper soon
and require a specific cohomological approach to introduce, have  better presentations without a too technical part on Cartan coverings, which will be only recalled and applied.

\section{Preliminary results}\label{prelim}
%\section{Matrix representation of quaternionic slice (regular) functions}

 Let
 $\mathbb{S}$ be the sphere of imaginary units in $\mathbb{H}$, i.e. the set of
 quaternions $I$ such that $I^2=-1$.  Given any quaternion $z \not\in
 \mathbb{R},$ there exist (and are uniquely determined) an imaginary
 unit $I$, and two real numbers $x,y$ ( with $y> 0$) such that
 $z=x+Iy$. With this notation, the conjugate of $z$ will be $\bar z :=
 x-Iy$ and $|z|^2=z\bar z=\bar z z=x^2+y^2$.
 Each imaginary unit $I$ generates (as a real algebra) a copy
 of a complex plane denoted by $\mathbb{C}_I:=\R+I\R$. We call such a complex
 plane a {\em slice}.  The upper half-plane in $\mathbb{C}_I$, namely $\{x+yI\ :\ y>0\}$ will be
 denoted by $\C_I^+$ and called a {\em leaf}. Set $\C_I^-:=\C_{-I}^+$ and for  a subset $E \subset \H$ define
 $E_I:=E \cap  \C_I$,  $E_I^+:=E \cap  \C_I^+$, $E_I^-:=E \cap  \C_I^-$, $E^0:=E \cap  \R.$
 % On each leaf $\C_I^+$ it is possible
 % to define the function $\arg_{I}:\C_I^+\to (0,\pi)$ as $z=x+Iy\in\C_I^+ \mapsto \cot^{-1} (x/y):=\arg_{I}(z)$.
 %If one consider the closure of $\C_I^+$, i.e. $\C_I^+\cup\R$, then
 % The function $\arg_I$ can be continuously extended
 % as a function $\arg_{I}:\C_I^+\cup\R^+\cup\R^-\to [0,\pi]$.
%   and $\arg_{I}:\C_I^+\cup\R^-\to (0,\pi]$.
%%
%%  It is also useful to define the imaginary unit
%%  function on $\H \setminus \R$ in the following way: if $z \in
%%  \C_I^+,$ i.e. if $z=x+Iy$, with $x,y\in\mathbb{R}$ and $y>0$, then
%%  $\cI (z) = I$. One can extend the above-given function $\cI$ over the
%%  real axis $\mathbb{R}$ by putting $\cI(z)=0$ if $z\in\mathbb{R}$, but
%%  the extended function is not continuous.
%%
\begin{definition}[Closure, interior, complex conjugation and complex symmetrization] \label{sym} For a set $D \subset \C$ we denote by $\overline{D}$ its closure, by
$\mathring{D}$ or $\mathrm{int}(D)$ its interior,  by $\refl{D}$ the conjugated (reflected) set,
$\refl{D} = \{\bar{z}, z \in D\}$ and by $S(D)$ the symmetrized set,  $S(D) = D \cup \refl{D}.$
 A set $D \subset \C$ is {\em symmetric} if $S(D) = D.$
An open set  $D \subset \C$ is {\em regular} if $\mathrm{int}(\overline{D})=D$.

A real-valued function $m$ on a symmetric set $D \subset \C$ is {\em symmetric} if $m(z) = m(\bar{z})$ for any $z \in D$.  We also use the notation $m(x,y)$ for  $m(z)$,  if $z = x + iy.$
\end{definition}
 Notice that for a smooth symmetric real-valued function defined near a point $(x_0,0)$, we always have $\grad m(x_0,0) =\lambda (1,0)$, $\lambda \in \R$.

% {\color{purple}
%\begin{definition}[$\ve$-neighbourhoods] Let $D \subset \C$ be a   set. Given $\ve > 0,$ an open set  $U = \{x \in \C, d(x,D) < \ve\}$  is an  {\em  $\ve$-neighbourhood} of $D$. \marginpar{Then we say that we take an $\ve$-neighbourhood with little wings attached when necessary. This seems to be the easiest way, without diffeotopies and retractions. $D$ is not necessarily bounded.}
%\end{definition}
%}

\begin{definition}[$\ve$-neighbourhoods] Let $D \subset \C$ be a bounded  set. Given $\ve > 0,$ we say that an open set  $U$, $ D \Subset U \subset \{x \in \C, d(x,D) < \ve\}$  is an
 {\em  $\ve$-neighbourhood} of $D$.
\end{definition}

The classical lemma below provides the existence of tubular neighbourhoods of simple closed curves and collars of closed arcs using neighbourhoods of zero sections in the normal bundle.
\begin{lemma} \label{difeo} Let $l:[0,1] \rightarrow \C$ be either a smooth arc or a smooth simple closed curve, $n: l^* \ra \C^*$  its (smooth) unitary normal, where $l^*:=l([0,1])$. Then there exists $r > 0$ so that $\phi_l: l^* \times [-r,r] \ra \C$ defined by $\phi_l(z,t) = z + n(z)t$ is a diffeomorphism onto
the image.
\end{lemma}
Notice that if $l^*$ is a closed curve, then the image of $\phi_l$ is an open neighbourhood of $l^*$, called a tubular neighbourhood. If $l^*$ is an arc, the image of $\phi_l$  is a double-sided collar but not a neighbourhood of $l^*$, since the endpoints of $l^*$ are on the boundary, so to get an open neighbourhood, one has to attach a suitable disc centred at each of the endpoints of $l^*$.

We want to extend the notion of the tubular neighbourhood to sets with piecewise smooth boundaries and closed arcs.
Therefore we provide the following definition.

%\begin{definition}[Tubular $\ve$-neighbourhoods] Let $D \subset \C$ be a bounded open set with piecewise smooth boundary {\color{purple} or a closed piecewise smooth arc}.
%We say that a set  $U$ with piecewise smooth boundary is a {\em tubular neighbourhood of $D$} if there exist a strong deformation retraction from $\overline{U}$ to $\overline{D}$.
%If, moreover,   $U$ is an $\ve$-neighbourhood of $D$  for a given $\ve > 0,$  we say that $U$ is a
% {\em tubular $\ve$-neighbourhood} of $D$.
%\end{definition}

\begin{definition}[Tubular $\ve$-neighbourhoods] Let $D \subset \C$ be a bounded regular open set with a piecewise smooth boundary or a closed piecewise smooth arc.
We say that  an open regular set  $U \supset D$ with a piecewise smooth boundary is a {\em tubular neighbourhood of $D$} if

(a) $\overline{U}$ and $\overline{D}$ are homeomorphic, if $D$ is an open set or

(b) $\overline{U}$ is homeomorphic to a closed topological disc if $D$ is an arc.

\noindent In addition, we require\footnote{Such  tubular neighbourhoods are analogous to {\em regular} neighbourhoods  as in J. H. C. Whitehead,
'Simplicial spaces, nuclei and m-groups', Proc. London
Math. Soc. {45} (1939) 243-327 } that $\overline{D}$ is a strong deformation retract  of $\overline{U}$.
If, moreover,   $U$ is an $\ve$-neighbourhood of $D$  for a given $\ve > 0,$  we say that $U$ is a
 {\em tubular $\ve$-neighbourhood} of $D$.
\end{definition}

\begin{remark} Let $\ve > 0$ be given. If $D$ is a closed topological disc with a piecewise smooth boundary, then we can enlarge it a little bit near the nonsmooth points of the boundary to obtain a smooth closed topological disc. Then, a suitable smooth collar attached to the boundary will provide the desired $\ve$-neighbourhood. If $D$ is a closed piecewise smooth arc, such an $\ve$-neighbourhood can be constructed by first smoothing the nonsmooth points, using the tubular neighbourhood obtained from the normal bundle and then attaching  a suitable disc centred at each of the endpoints
of the arc.
\end{remark}

\begin{definition}
The
 \emph{(axial) symmetrization} $\widetilde E$ of a subset $E$ of $\bH$ is defined by
 $$\widetilde{E} = \{ x + I y:  x,y \in \R, I\in \bS, (x + \bS y) \cap E \neq \varnothing \}.$$
If $E=\{q\}$, we write $\widetilde{q}$ for the set  $\widetilde{\{q\}}$.

 A subset $\Omega$ of $\bH$ is called
 {\em (axially) symmetric} (in $\bH$) if $\widetilde{\Omega} = \Omega.$
\end{definition}

%Let $\Omega$ be an axially symmetric domain of $\H$, i.e.
%an open and connected set in $\H$ such that if $z=x+Iy\in \Omega$, then $x+Jy\in \Omega$ for any $J\in\mathbb{S}$ .

\begin{proposition}
Let $\Omega\subseteq \H$ be an axially symmetric domain. For all $I\in \mathbb S$, we have that \[
\Omega=\bigcup_{x+Iy\in \Omega_I} x+\mathbb Sy
\]
Moreover, for all $I\in \mathbb S$, the set $\Omega_I\subseteq \C_I$ is invariant under conjugation, i.e., $\Omega_I=\refl{\Omega_I}$.
\end{proposition}

The following definition introduces a class of natural domains of definition for slice--regular functions.

\begin{definition}\label{slice_domain} A domain $\Omega$ of $\mathbb H$ is called a \emph{slice domain} if, for all $I\in \mathbb S$, the subset
$\Omega_I$ is a domain in $\C_I$ and if $\Omega^0=\Omega\cap\R\neq \varnothing$.
If, moreover, $\Omega$ is axially symmetric, then it is called a {\em symmetric} slice domain.
\end{definition}

On the other hand, slice functions (see [GP]) are naturally defined  on axially symmetric domains which are not necessarily slice domains.
%on a different class of domains.

\begin{definition} An axially symmetric domain $\Omega$ of $\mathbb H \setminus \R$ is called a \emph{product domain}.
\end{definition}

%\begin{remark}
%Let $\Omega\subseteq \H$ be
Hence, an axially symmetric domain $\Omega$ is either a symmetric slice domain or a product domain.
%\end{remark}

If $\Omega\subseteq \H$ is an axially symmetric domain, then for (one and hence for) all $I\in \mathbb S$, the set $\Omega_I$ is an open subset of $\C_I$ such that either it is a connected set that intersects $\R$, or it has two symmetric connected components separated by the real axis, swapped by the
conjugation. In the former case, $\Omega$ is an axially symmetric slice domain; in the latter case, $\Omega$ is a product domain.
\begin{definition} An axially symmetric domain $\Omega$  has {\em a slice--piecewise smooth boundary} if for some (and hence for all) $I \in  \bS,$ the set
$\Omega_I \subset \C_I$  has a piecewise smooth boundary.
\end{definition}

The following classes of domains will play a key role in this paper.

\begin{definition}\label{gjf}
  An axially symmetric domain $\Omega$ of $\mathbb H$ is called an {\em (open) basic set or a basic domain} if, for (one and hence for) all $I\in \mathbb S$, the single connected component or both the connected components of  $\Omega_I$ are simply connected.
An open basic set is also a \emph{basic neighbourhood} of any of its points.  We also  define the empty set to be an open basic set.
 An axially symmetric closed set  $V$ of $\mathbb H$ is called a {\em closed basic set}  if, for (one and hence for) all $I\in \mathbb S$, the set $V_I$ has either a single connected component if it intersects the real axis or has two connected components otherwise, and in both cases the connected components of  $V_I$ are closed topological discs.
\end{definition}

A closed basic set intersecting the real axis is a closed topological ball (i.e. homeomorphic to a closed ball). Notice that the intersection of a basic domain with the real axis is either empty or connected. The closure of a basic set is not necessarily a closed basic set. For example, the sets $\Omega_1:=\{ x + Iy, y > 0, x^2 + (y-1)^2 < 1, I \in \bS\}$ and $\Omega_2:=\{ x + Iy, y > 0, x \in (-1,1), (1-x^2)/2 < y < \sqrt{1-x^2}, I \in \bS\}$ are open basic sets (and product domains) while their closures are not closed topological balls. A closed basic set with a slice--piecewise smooth boundary has a basis of basic sets.

 The interested reader can find the standard definition of slice--regular functions and their properties in \cite{GMP, GSS, AdF1}. Here, we present an equivalent way of defining the set of slice--regular functions on axially symmetric domains, together with summation and $*$-multiplication, which is easier for a nonexpert reader. For the equivalence between the definitions of slice--regularity, we refer the reader to Lemma 6.11 in \cite{GMP}.

\begin{definition}[Slice--regular functions]\label{SR_simplified} Let $\Omega$ be an axially symmetric domain, $I \in \mathbb{S}$ and $f: \Omega_I \rightarrow \C_I$
a Schwarz symmetric holomorphic function, i.e. $f(x + Iy) = u(x,y) + I v(x,y) = \overline{f(x - Iy)},$ $u = \Re(f),$ $v = \Im(f)$. The extension of $f$ from $\Omega_I$ to $\Omega$, defined by $f(x + Jy) = u(x,y) + J v(x,y)$
for any $J \in \mathbb{S}$, is a {\em slice--preserving slice--regular function}. We denote the set of all such functions by $\mathcal{SR}_{\R}(\Omega)$.
Let $\{1,i,j,k\} \subset \mathbb{H}$ be a standard basis of $\mathbb{H}$.  The set of all {\em slice--regular functions} on $\Omega$ is
\[\mathcal{SR}(\Omega):=\{f_0 + f_1i + f_2j + f_3k, f_0,\ldots,f_3 \in \mathcal{SR}_{\R}(\Omega)\}.\] The set  $\mathcal{SR}^*_{\R}(\Omega) \subset \mathcal{SR}_{\R}(\Omega)$ denotes the subset set of all nonvanishing functions and
 $\mathcal{SR}^+_{\R}(\Omega)\subset \mathcal{SR}_{\R}(\Omega)$ denotes the set of all functions from $\mathcal{SR}_{\R}(\Omega)$ which are strictly positive on the real axis, provided that $\Omega \cap \R \ne \varnothing$.
\end{definition}

 The set $\mathcal{SR}(\Omega)$ is equipped with the topology of uniform convergence on compact sets. For a compact set $K \subset \Omega$ and $f \in \mathcal{SR}(\Omega)$ we define
 $|f|_K:= \max\{|f(q)|, q \in K\}.$ The sets $U(f,K, \ve) = \{g \in \mathcal{SR}(\Omega), |f - g|_K < \ve \}$ define a basis for this topology.

\begin{remark} Definition \ref{SR_simplified} immediately implies that any slice--regular function is uniquely defined by its restriction to any slice and vice versa. Given
$f_0,\ldots, f_3 \in \mathcal{O}(\Omega_I),$ there exists a unique extension of $f_0 + f_1 i + f_2 j + f_3 k$ to $\Omega$.
\end{remark}
Let us now define the \emph{imaginary unit function}
 \[\cI \colon \H \setminus \R \to \mathbb{S}\, \]
by setting $\cI(q) = I$ if $q \in \mathbb{C}_I^+$.  The function $\cI$ is slice--regular and slice--preserving because it is an extension of the function
defined as $f\equiv I$ on $\C_I^+$ and $f\equiv -I$ on $\C_I^-$,  but it is not an open mapping, and is not defined on any slice domain. %{\color{blue} We extend the definition of $\mathcal{I}$ to $\R$ by $\mathcal{I}(x) = 0$ for all $x \in \R$.}
\begin{definition}[The sum and the $*$-product]
Given any $f, g \in \mathcal{SR}(\Omega)$,
$f = f_0 + f_1 i + f_2 j + f_3 k$, $g=g_0+g_1i+g_2j+g_3k$,
we define the sum as $f+g:= (f_0+ g_0) + (f_1+g_1) i + (f_2 + g_2) j + (f_3+g_3) k$.  The $*$-product of $f$ and $g$ is defined as
$f*g:= (f*g)_0+(f*g)_1i+(f*g)_2j+(f*g)_3k,$ where
\begin{equation}\label{**}
  \begin{array}{ccc}
    (f*g)_0&=& f_0g_0-f_1g_1-f_2g_2-f_3g_3\\
    (f*g)_1&=& f_0g_1+f_1g_0+f_2g_3-f_3g_2\\
    (f*g)_2&=& f_0g_2-f_1g_3+f_2g_0+f_3g_1\\
    (f*g)_3&=& f_0g_3+f_1g_2-f_2g_1+f_3g_0
    \end{array}
\end{equation}
\end{definition}
Notice that the definition of the $*$-product mimics the usual product in quaternions. It is associative.
\begin{proposition}\label{regularitygp1} Let $\Omega \subseteq \H$ be an axially symmetric open set, and let $f, g\in \mathcal{SR}(\Omega)$ be two slice--regular functions. Then
\begin{enumerate}
\item[(a)] the $*$-product $f*g$ is a slice--regular function on $\Omega$  and $(f*g)(q)$ is either $0$ if $f(q) = 0$ or else $(f*g)(q)= f(q) g(f(q)^{-1}qf(q))$;
\item[(b)] if $f$ is slice--preserving, then $f*g=fg=g*f$, i.e, the $*$-product coincides with the pointwise product;
\item[(c)] if $f$ is slice--preserving, then $g \circ f$ is slice--regular;
\item[(d)]   if $K \subset \Omega$ is compact and axially symmetric, then $|f *g|_{K} \leq |f|_K |g|_K$, therefore the $*$-product is continuous in the topology of uniform convergence on  compact sets.
\end{enumerate}
\end{proposition}
\begin{proof} For claims (a) -- (c) we refer the reader to \cite{GSS}. For (d), observe that by (a), we have $|(f*g)(q)|= |f(q)|| g(f(q)^{-1}qf(q))|$ and hence
$|(f*g)(q)|\leq  |f(q)| |g|_{\widetilde{q}}$. This implies that also $|f*g|_{\widetilde{q}}\leq  |f|_{\widetilde{q}} |g|_{\widetilde{q}}$ and so $|f *g|_{K} \leq |f|_K |g|_K$ for axially symmetric compact sets. Since axially symmetric compact sets exhaust $\Omega$, the $*$-product continuity follows from this estimate.
\end{proof}

\begin{remark}
Recall (see \cite{AdF1, AdF2, GPV}) that for any $f \in \mathcal{SR}(\Omega)$ the function $\exp_*f:= \sum_{n=0}^{\infty} f^{*n}/n!$ is slice--regular on $\Omega$ and if $|1-f|_{\Omega}\leq r < 1$, also  $\log_*(1-f):=-\sum_{n=1}^{\infty}f^{*n}/n$ is slice--regular and $\exp_*(\log_*(1-f))=1-f.$ In contrast with the complex case, there are nonvanishing  functions on balls in $\H$ without slice--regular logarithm.
\end{remark}

\section{Properties of slice--regular functions on finite disjoint unions of basic sets}\label{group}

We begin the section with this technical lemma.

\begin{lemma}\label{retraction} Let $V \subset \bH$ be a closed basic set, $B = \mathring{V}$ its interior. Then there exists a
homotopy  of slice--preserving slice--regular mappings with $H(\cdot,1)$ the identity mapping and $H(\cdot,0)$ a slice--preserving retraction $B \rightarrow \widetilde{q}$,
 where $q \in \R$ if $B$ is a slice domain and $q \not \in \R$ if $B$ is a product domain.
\end{lemma}

\begin{proof}

If $B$ is a product domain, choose  $q = a + ib \in B$. Then  $\widetilde{q}=a + b \bS \subset B$ is a sphere. Denote by $\Delta$ the unit disc in $\C_i$ centred at $2i$.
By Riemann mapping theorem there exists a biholomorphic mapping $f_i:B^+_i \ra \Delta.$  Set $f_i(z):= \overline{f^+_i(\bar{z})}$ on $B_i^-.$ Then the extension of $f_i$ to $\widetilde{\Delta}$ is slice--regular slice--preserving and
such that  $f(a + Ib) = 2I$ for each $I \in \bS$.
Set $H_1(q,t) = 2\mathcal{I}(q)+ t(q-2\mathcal{I}(q))$ on  $\widetilde{\Delta} $. %$\widetilde{B^4(2i,1)}$.
Then $H:= f^{-1} \circ H_1 \circ f$
is the desired homotopy.\\

 If $B$ is a slice domain choose $q \in B \cap \R$, so  $\widetilde{q} = \{q\}$.
Fix $I \in \bS$  and let  $\mathbb{D}$ be the unit disc. Put $\mathbb{D}^+:=\{ z\in\C_I : |z|<1, \Im (z)\geq0\}$, $\mathbb{D}^-:=\{ z\in\C_I : |z|<1, \Im (z)\leq 0\}$, $(\partial B_I)^+:= \partial B_I \cap \C_I^+$ and notice that since $\overline{B_I} = V_I$ is a closed topological disc, also the set $\overline{B_I^+}$ is a closed topological disc.

The claim follows from the fact that there exists, by Riemann mapping theorem, a biholomorphic mapping $f_I:B_I \ra B^2(0,1),$ such that $f_I(\bar{z}) = \overline{f_I(z)}$.

To see this, let $\vph: B_I^+ \ra \mathbb{D}$ be a biholomorphism. Because $\partial B_I^+$ is a Jordan curve, the mapping extends to a homeomorphism  $\vph: \overline{B_I^+} \ra \overline{ \mathbb{D}}$  and hence $\overline{\vph(B^0)}$ is a closed arc in the unit circle. Therefore $\vph(B_I^+ \cup B^0)$ can be mapped  to $\mathbb{D}^+:=\{ z\in\C_I : |z|<1, \Im (z)\geq0\}$ via a Riemann mapping $\psi$,  so that $(\psi)^{-1}(-1,1) = \vph(B^0)$. The mapping
$f_I^+:= \psi \circ \vph$,
extended to $\mathbb{D}^-:=\{ z\in\C_I : |z|<1, \Im (z)\leq 0\}$ by reflection,  defines the mapping $f_I$ with the required properties.
The  extension of $f_I$ to $B$  slice--preserving slice--biregular mapping $f:B \rightarrow B^4(0,1) \subset \bH, $ such that $ f(q) = 0.$
 Set $H_1(q,t)=tq $ on $B^4(0,1)$.  Then $H:= f^{-1} \circ H_1 \circ f$
is a homotopy  of slice--preserving slice--regular mappings with $H(\cdot,1)$ the identity mapping and $H(\cdot,0)$ a slice--preserving retraction $B \rightarrow \{q\}$.
\end{proof}

\begin{proposition}\label{ConnGroup} Let $B \subset \bH$ be a basic domain  with $\overline{B}$ a closed basic set. Then, the group of nonvanishing slice--regular functions
$(\mathcal{SR}^*(B),*)$ is connected.

 In addition, if $B$ is a slice domain, then the group $(\mathcal{SR}^*_{\R}(B),*)$ has two connected components and if $B$ is a product domain then the group $(\mathcal{SR}^*_{\R}(B),*)$ has one connected component.
\end{proposition}
\begin{proof}

If $B$ is a slice domain, Lemma \ref{retraction} shows that $F = f \circ H$ is a homotopy
through nonvanishing slice--regular functions between $f$  and a nonzero constant, which is homotopic to the constant $1$ through nonzero constants.

   If  $B$ is a product domain and $f$ a nonvanishing slice--regular  function, then Lemma \ref{retraction} gives a homotopy
    $F = f \circ H$ through nonvanishing slice--regular functions between $f = F(\cdot, 1)$ and $F(q, 0) = f(a + \mathcal{I}(q)b) = q_0 + \mathcal{I}(q)q_1$ which is  nonvanishing.

     If $q_1 = 0,$ we are done. If not, we can write
     $q_0 +  \mathcal{I}q_1 = (q_0q_1^{-1} +  \mathcal{I})q_1$ and by using a homotopy between $q_1$ and $1$ in $\bH \setminus\{0\},$ we may assume that $F(q, 0)$ is of the form  $F(q, 0)= q_0 +  \mathcal{I}(q).$  Consider a homotopy  $h_t(q): q \mapsto q_0 +  \mathcal{I}(q)t$.
     If $h_t(q) \ne 0$, then we are done. Otherwise there exist $p,t_0$ so that $q_0 + \mathcal{I}(p)t_0 = 0$ and hence $q_0$ is purely imaginary.  For any $\ve \in (0,\infty)$, the map $q \mapsto q_0 +  \mathcal{I}(q)$ is homotopic to $q \mapsto q_0 + \ve +\mathcal{I}(q)$ through nonvanishing functions and then  the homotopy $\chi_t(q)= q_0 + \ve + \mathcal{I}(q)t$ is nonvanishing with $\chi_0(q) = q_0 + \ve \ne 0$.
     Clearly, $\chi_0$ is homotopic to $1$ through nonzero constants.

        Assume now $f \in \mathcal{SR}^*_{\R}(B)$.
          If  $B$ is a slice domain, then  either $f(B \cap \R) \subset (0,\infty)$ or $f(B \cap \R) \subset (-\infty,0)$. Since slice--preserving functions map the real axis to itself, slice--preserving functions $f$ and $g$ with $f(B \cap \R) \subset (0,\infty)$ and $g(B \cap \R) \subset (-\infty,0)$  cannot be connected through a homotopy of nonvanishing slice--preserving functions.
           Since $B$ is a slice domain, then the above $F$  from Lemma \ref{retraction} connects $f$ with $F(q, 0) = f(a)$  and the value $f(a)$ is a nonzero real number. Because $\R \setminus \{0\}$  has two connected components, also $(\mathcal{SR}^*_{\R}(B),*)$ has two connected components.

    If $B$ is a  product domain, then $f$ is homotopic to $F(q, 0) = f(a + \mathcal{I}(q)b) = q_0+ \mathcal{I}(q)q_1$ with $q_0, q_1$ real. A similar argument as in the first part of the proof provides a homotopy between this map and constant $1$.
\end{proof}

\begin{proposition}\label{approx}(Compare \cite{GR}, VI.E, Lemma 2).
  Let  $B$ be a basic domain such that $\overline{B}$ is a closed basic set and $K \subset B$  a compact set. %(with respect to the origin $0$) in $\H$.
  Then for any $f\in \mathcal{SR}^*(B)$
  %  $f\in\mathcal{S}\mathcal{R}(\Omega)$ and $f^s$ does not vanish in $\Omega$,
  there exist ${f_{1}}, {f_{2}},\ldots, {f_{n}}\in{\mathcal{SR}^*}(B)$ such that%
  \[f={f_{1}}*{f_{2}}* \cdots * {f_{n}} \mbox{ in } B \]
  and $|{f_{j}}-1|_{\widetilde{K}}<1$ for $j=1,2,\ldots, n$.
\end{proposition}

\begin{proof} Because $\mathcal{SR}^*(B)$ is a (path) connected topological group (equipped with the topology of uniform convergence on compact sets), every neighbourhood of $1$ generates the whole group. Since   there exists a homotopy between $f \in \mathcal{SR}^*(B)$ and the constant $1$ and its image is compact,  there exist functions $f_1,\ldots f_n$ from the neighbourhood $U:=\{g, |g-1|_{\widetilde{K}}<1\}$ of $1$ so that   $f = f_1* \ldots *f_n$.
\end{proof}
\begin{remark} Proposition  \ref{approx} also holds for the interior of a finite union of disjoint closed basic sets.
\end{remark}

\begin{problem} We do not know whether a bounded slice--regular function $f$ on a basic domain $\Omega$ admits a (finite) factorization
to slice--regular functions $f_j$ satisfying $|f_j - 1|_{\Omega} < 1.$ The positive answer could be of interest to prove the contractibility of the subgroup of bounded functions in $\mathcal{SR}^*(\Omega)$.% and $\mathcal{SR}^*_{\omega}(\Omega)$.
\end{problem}

We finish this section by proving Runge-type approximation results; see also  \cite{bw} for Theorem \ref{RungeR} and  \cite{GR}, VI.E, Theorem 3, for Theorem \ref{RungeInvF}.

\begin{theorem}[Runge Theorem for $\mathcal{SR}_{\R}$ and $\mathcal{SR}$] \label{RungeR}
    Let $ U \subset \Omega$ be axially symmetric domains, $K \subset U$ a compact set and $f \in \mathcal{SR}(U)$.  If for some $I \in \mathbb S$ the set  $U_I$ is Runge in $\Omega_I$ and if $\ve > 0,$ then there exists $\tilde{f} \in \mathcal{SR}(\Omega)$ so that $|f- \tilde{f}|_K < \ve.$ If, in addition, $f \in \mathcal{SR}_{\R}(U)$, then we can also choose $\tilde{f} \in \mathcal{SR}_{\R}(U).$
\end{theorem}
\begin{proof}
 Assume that $f$ is slice--preserving. The restriction  $f|_{U_I}$  is holomorphic, then, by classical results, it can be approximated on ${K_I}$ by a holomorphic function  ${g} \in \mathcal{O}(\Omega_I)$ as well as we wish. The function $\tilde{f}(z):=(g(z) + \overline{g(\bar{z})})/2$  also approximates $f|_{U_I}$ on $K_I$ and extends to a slice--preserving function on $\Omega$. Because by definition of slice--regularity, for $f \in \mathcal{SR}(U)$ we have $f = f_0 + f_1i + f_2j + f_3k$ with $f_0,\ldots,f_3$ slice--preserving, the Runge theorem also holds for slice--regular functions.
 % Its unique extension to $\Omega$ gives the desired approximation.
\end{proof}
\begin{remark}\label{remark37} Because in $\C$, any holomorphic function defined on an open set containing a finite union $V$ of disjoint closed discs in $\C$ can be approximated uniformly on $V$ by entire functions and hence by holomorphic polynomials,  also any slice--regular function defined on an open neighbourhood of a finite  union $V$ of disjoint closed basic sets  can be approximated uniformly on $V$ by slice--regular polynomials. %In particular, such $V$ is also Runge in any axially symmetric open set $\Omega \supset V$.
\end{remark}

\begin{theorem}[Runge theorem for $\mathcal{SR}^*$]\label{RungeInvF}
Let $K$ be a union of finitely many disjoint closed basic sets in $\H$ and let $f\in \mathcal{SR}^*(\Omega)$ with $\Omega$ an open axially symmetric neighbourhood of $K$.
Then, for any $\varepsilon>0$, there exists $g\in {\mathcal{SR}^*}(\H)$ such that
\[ |f-g|_{K}<\varepsilon.\]
\end{theorem}

\begin{proof}
  By Proposition \ref{approx}, for $j=1,\ldots, n$,
  one can define $h_{j}:=\log ({f_{j}})$ on a  suitable neighbourhood $U$ of $K$ in $\Omega$.
By Remark  \ref{remark37} there exist slice--regular polynomials $p_{j}$  such that
$|h_{j}-p_{j}|_K$ can be made as small as desired.  Following the proof in \cite{GR} we see that the function  $g:=g_{1}* g_{2} *\cdots * g_{n},$ where $g_{j}=:\exp_*(p_{j}),$ $j=1,\ldots, n,$ fulfills the requirements of the theorem.
%%
%%Now, since the exponential mapping is continuous, we have that if $G_{\alpha_j}=:\exp (P_{\alpha_j})$
%%with $G_{\alpha_j}\in\mathcal{GL}_{\mathcal S}(\H)$ for $j=1,\ldots, n$, and
%%%
%%\[ |M_{f_{\alpha_j}}-G_{\alpha_j}|_K<\varepsilon':=\dfrac{\varepsilon}{n C^{n-1}}.\]
%%%
%%Moreover,
%%%
%%\[ |G_{\alpha_j}|_K\leq |G_{\alpha_j}-M_{f_{\alpha_j}}|_K+|M_{f_{\alpha_j}}|_K<(C-1)+1=C.\]
%%Hence,
%%%
%%\[ G:=G_{\alpha_1}\cdot G_{\alpha_2} \cdots G_{\alpha_n}\]
%%belongs to $\mathcal{GL}_{\mathcal{SR}}(\H)$ and is such that
%%%
%%\begin{eqnarray*}
%%  |M_f-G|_K&=&|M_{f_{\alpha_1}}\cdot M_{f_{\alpha_2}}\cdots M_{f_{\alpha_n}}- G_{\alpha_1}\cdot G_{\alpha_2} \cdots G_{\alpha_n}|_K\\
%%  &\leq& |M_{f_{\alpha_1}}\cdot M_{f_{\alpha_2}}\cdots M_{f_{\alpha_n}}-M_{f_{\alpha_1}}\cdot M_{f_{\alpha_2}}\cdots M_{f_{\alpha_{n-1}}}\cdot G_{\alpha_n}|_K+\ldots\\&
%%  &\ldots+|M_{f_{\alpha_1}}\cdot G_{\alpha_2}\cdots G_{\alpha_n}-G_{\alpha_1}\cdot G_{\alpha_2}\cdots G_{\alpha_{n-1}}\cdot G_{\alpha_n}|_K\\
%%&\leq& C^{n-1}|M_{f_{\alpha_n}}-G_{\alpha_n}|_K+\ldots + C^{n-1}|M_{f_{\alpha_1}}-G_{\alpha_1}|_K<\varepsilon.
%%\end{eqnarray*}
\end{proof}

\section{Cartan coverings}\label{coverings}

To proceed towards  Cousin problems in the framework of slice--regular functions, one
needs to define a special type of axially symmetric coverings of
axially symmetric open sets $\Omega \subset \H;$ without loss of
generality we will assume that $\Omega$ is an axially symmetric
domain.
 We have seen in the previous section that for good approximation properties, the sets in question have to be finite disjoint unions of closed basic sets.

The assumptions on symmetry allow us to construct the open covering in the complex plane and then extend it to quaternions by symmetrization. The construction
in the plane resembles, in part, the construction of Lebesgue; Lebesgue's `bricks' are symmetric tiles in our setting, with additional properties for the tiles intersecting the real axis. Moreover, to be able to proceed inductively, the elements of the covering have to be ordered in such a way that for any $n$ % {\color{green} the union of  the first $n$ elements of the covering is a symmetric  Runge????? set in $\Omega$ and - ALL THIS OUT!!}
the intersection of the $n$-th tile with the union of the previous ones is a finite union of disjoint closed basic sets with piecewise smooth boundaries.

\subsection{Symmetric  exhaustions and symmetric coverings}%\newline

A symmetric covering of a symmetric set is defined as expected.
\begin{definition} Let $D \subset \C$ be a  symmetric  open set and
  let $\mathcal {D}=\{D_{\lambda}\}_{\lambda \in\Lambda}$
  be an open
covering of $D.$ The covering $\mathcal{D}$
is  a
{\em  symmetric open
  covering} of $D$ if each $D_{\lambda} \in {\mathcal D}$ is
symmetric.
\end{definition}

We recall the following
\begin{definition}
  Given any  open set  $D \subset \C$,
    the sequence $\{K_n\}_{n \in
      \bN_0}$ of compact sets is called an {\em exhaustion of $D$ with
      compact sets} if $K_n \Subset K_{n+1}$ and $\cup_{n \in \bN_0} K_n
    = D.$
    If the set $D$ and the sets $\{K_n\}_{n \in \bN_0}$ are symmetric, then the exhaustion is called
    a {\em symmetric exhaustion} of $D.$  If, moreover, $\{K_n\}_{n \in \bN_0}$ are Runge in $D$, the exhaustion is called
    a {\em symmetric Runge exhaustion} of $D.$
\end{definition}

\noindent{\bf Notation.} Let $K \subset \C$ be homeomorphic to a closed $k$-annulus. Then there exists a finite family of disjoint
closed topological discs $D_1,\ldots, D_k$, called {\em holes}, with interiors disjoint from $K$ so that {\em the filled $K$,} $K^{\bullet}:=K \cup (\cup_{j=1}^k D_j)$ is simply connected.
If $K = K_1 \cup \ldots \cup K_m$ is a union of disjoint compact sets $K_j,$ each one homeomorphic to a closed $k_j$-annulus, then the filled $K$ is  $K^{\bullet} = K_1^{\bullet} \cup \ldots \cup K_m^{\bullet}.$  We also set $\varnothing^{\bullet} :=\varnothing.$

\begin{proposition}\label{SymCov} Let $D \subset \C$ be a symmetric set and $K \subset D$ a symmetric compact set with smooth boundary, that is Runge in $D$. Then there exists a symmetric Runge exhaustion of $D$ with compact sets $\{K_n\}_{n \in \bN_0}$ such
that $\partial K_n$ is smooth for each $n$ and $K_0 = K$ if $K \ne \varnothing.$ If $K = \varnothing,$ then we choose the set $K_0 \subset D$
 to be either a closed symmetric topological disc or $K_0 = D \cup R(D),$ where $D$ is a closed topological disc in the upper half-plane.
 % and if
%$D \cap  \R \ne \varnothing$, then also $K_0 \cap \R \ne \varnothing$.
\end{proposition}

\begin{remark} Since $ K_n$ is compact with a smooth boundary, it has a finite number of connected components. For each connected component $K$ of $K_n$, there exists a $k\in \bN_0$ so that  $K$ is homeomorphic to a closed $k$-annulus.
\end{remark}

\begin{proof}

Consider first the case when $D \cap \R = \varnothing$  and assume that $K =\varnothing$. Then the
set $D^+:= D \cap \C^+$ is Stein, and by classical results, there exists a strictly
subharmonic exhaustion function $m : D^+ \ra [0,\infty)$ with
  only nondegenerate critical points and with a global minimum $0$
  attained at precisely one point. Since the set of regular values is an open set, there exists a strictly increasing sequence of regular values numbers $\{r_n\}_{n \in \bN_0}.$  Set $K^+_n:= m^{-1}((-\infty, r_n]).$ Then the sets $K_n:=K^+_n \cup \refl{K^+_n}$  have all the desired properties provided $r_0$ is small enough.

  If $K = K_0 \ne \varnothing$ is already given, then we choose $r_1$ to be such that $K_0 \Subset K_1$.

 If $D \cap \R \ne \varnothing,$  then let $m_1: D \ra [0,\infty)$
   be a strictly subharmonic exhaustion function with only nondegenerate critical points and with a global minimum $0$
  attained at precisely one point, $m_1^{-1}(0) = (x_0,0).$ Then the function $m: D \ra [0,\infty)$ defined by $m(x,y) = m_1(x,y) + m_1(x,-y)$ is  a symmetric  strictly subharmonic exhaustion function with $m^{-1}(0) = (x_0,0)$ a global minimum and
% for $m_1$  the point $(x_0,0)$ has to be nondegenerate; in this case the hessian of m at $(x_0,0)$ is diagonal with $2m_{xx}$ and $2m_{yy}$ on the diagonal, which are positive}
 $(x_0,0)$ a nondegenerate critical point.
%\marginpar{easy to achieve by adding a function $|x - x_0|^2 + y^2$}
    If $\{r_n\}_{n \in \bN_0}$ is  a strictly increasing sequence of regular values, then the sets
  $K_n:=m^{-1}([0, r_n]), n \in \bN_0$ are Runge in $D$. If $K = K_0$ is already given, then we choose $r_1$ to be such that $K_0 \Subset K_1$. If $K_0$ is not given, we choose $r_0$ to be close enough to $0$ and then $m^{-1}([0,r_0])$ is a closed symmetric topological disc.

 %Fix $n \in \mathbb N_0$.  Let $\ve_n$ be such that $r_{n-1} < r_n - \ve < r_n + \ve < r_{n+1}$ and $\grad m \ne 0$ on $m^{-1}([r_n - \ve ,r_n + \ve]).$ We approximate $\partial Q_n$ with piecewise linear continuous simple closed curves which intersect the real axis transversely. These curves then define a compact set $K_n$, which is also Runge. Define the sets $K^+_n:= m^{-1}((-\infty, r_n])\cap (\C^+_I \cup \R)$ and
%   $K'_n:=K^+_n \cup \refl{K^+_n}.$  The sets $K_n'$ are symmetric, but $\partial K_n'$  is only  piecewise smooth with a finite number of transverse intersections.
%   However, $\partial K_n'$ can be  locally smoothened in a symmetric way thus yielding a symmetric compact set $K_n$ with smooth boundary.
%   The sets are not necessarily Runge, but we can make them Runge by adding those connected components of $\C\setminus K_n',$ which do not contain points from the complement of $D.$ Because $K_n^+$ contains $Q_{n-1}^+:=Q_{n-1} \cap (\C^+_I \cup \R),$ the sequence $\{K_n\}_{n \in \bN}$ is also exhaustive.
%
  \end{proof}

\subsection{Axially symmetric  coverings and symmetric tilings}

In this section, we introduce the coverings of axially symmetric domains we are looking for induced by tilings.

\begin{definition}\label{asop} Let $\Omega \subset \H$ be an axially symmetric domain and
  let $\mathcal {U}=\{U_{\lambda}\}_{\lambda \in\Lambda}$
  be an open
covering of $\Omega.$ The covering $\mathcal{U}$
is  an
{\em axially symmetric open
  covering} if each $U_{\lambda} \in {\mathcal U}$ is an axially
symmetric open set.

A covering $\mathcal{U}$
of $\Omega$ is a {\em basic covering} if each $U_{\lambda}
\in {\mathcal U}$ is a basic  set.

An indexed family $\mathcal{U}'=\{U'_{\lambda'}\}_{\lambda' \in\Lambda'}$ of subsets of $\Omega$
is said to be {\em subordinated} to the covering $\mathcal{U}$ of $\Omega$
if for each $U'_{\lambda'}\in\mathcal{U}'$ there exists $U_{\lambda} \in \mathcal U$ such that
$U'_{\lambda'} \subset U_{\lambda}$.\\
\end{definition}
%Since $\Omega \subset \H$ is an axially symmetric domain of $\mathbb{H}$,
%for any $I\in\mathbb{S}$ the set $\Omega_I=\Omega\cap
%\mathbb{C}_I$ is an open set domain in $\mathbb{C}_I$;
If  $\Omega \subset \H$ is an axially symmetric domain of $\mathbb{H}$,
$\mathcal{U} =\{U_{\lambda}\}_{\lambda \in\Lambda}$  a covering of $\Omega$,
then $\{U_{\lambda}\cap \mathbb{C}_I\}_{\lambda \in\Lambda}$ is a symmetric
covering of $\Omega_I$ which will be denoted by $\mathcal{U}_I$.
For any given indexed family $\mathcal {V}=\{V_{\lambda}\}_{ \lambda
\in\Lambda}$ of sets we indicate
%\marginpar{indicate instead of describe?}
their intersections by using the
following standard notation: if $(\lambda_1,\ldots,\lambda_k)
\in \Lambda^k$ is a multiindex, then
$V_{\lambda_1 \ldots \lambda_k}:= V_{\lambda_1} \cap \ldots \cap V_{\lambda_k}.$\\

Definition \ref{CartanPair} is an adaptation of {\em Cartan strings} (see Section 4 in \cite{fp} or subsections 6.7 - 6.9 in \cite{FF}) to axially symmetric domains in $\H$.

\begin{definition}[Cartan pair, Cartan string]\label{CartanPair}
Let $A, B \subset  \bH$ be axially symmetric compact sets with  slice--piecewise smooth boundaries fulfilling the {\em separation property} $\overline{(A \setminus B)} \cap \overline{(B \setminus A)} = \varnothing.$ If $B$ is a closed basic set and $A \cap B$ a finite union of disjoint closed basic sets
 then we say that the pair $(A,B)$  is a {\em Cartan pair or a Cartan $2$-string}.  A sequence  $(A_1,\ldots, A_n) $ of axially symmetric  compact sets with piecewise smooth boundaries contained in $\bH$ is a Cartan $n$-string if $(A_1,\ldots, A_{n-1})$ and $(A_1 \cap A_n,\ldots A_{n-1} \cap A_n)$ are Cartan $(n-1)$-strings and $(A_1\cup \ldots \cup A_{n-1}, A_n)$ is a Cartan pair.
\end{definition}

\begin{definition}[Cartan sequence]\label{CartanSequence}
Let $\Omega$ be an axially symmetric domain.
Let $\mathcal {A}=\{A_n\}_{n \in \bN_0}$ be a sequence of closed basic sets in $\Omega$ % which form a covering of $\Omega$  subordinated to ${\mathcal U}$
  such that
\begin{itemize}
\item[$(1)$] for all $n \in \bN_0$  the sets $A_{n}$ have slice--piecewise smooth boundary, %we have $\partial B_{n,I} =\partial \overline{}(B_{n,I}) $ and these  sets are  piecewise smooth;
%\item[$(2)$]  if $A_{n_i} \cap \R \ne \varnothing$ for $i \in \{1,2\}$ with $n_1\ne n_2$ and $A_{n_1n_2} \ne \varnothing$, the set $A_{n_1n_2}$ is a slice domain and a basic domain;  moreover, if $A_{n_i} \cap \R \ne \varnothing$ for distinct $n_1, n_2,n_3$, then  $A_{n_1n_2n_3} = \varnothing$;
 \item[$(2)$] $A_{n_1n_2}$
and  $ A_{n_1n_2n_3}$
are  %finite union of disjoint
closed basic sets for distinct $n_1,n_2,n_3 $;
 \item[$(3)$]  if $A_{n_i} \cap \R \ne \varnothing$ for $i \in \{1,2\}$ with $n_1\ne n_2$ and $A_{n_1n_2} \ne \varnothing$,
 %the set $A_{n_1n_2}$ is a closed basic set with
 then $A_{n_1n_2}\cap \R \ne \varnothing$; %and a basic domain;
moreover, if $A_{n} \cap \R \ne \varnothing$ there exists at most two indices $n_1 \ne n_2$ different from $n$
with $A_{n_i} \cap \R \ne \varnothing$, $i = 1,2$, such that
%, then  $A_{n_1n_2n_3} = \varnothing$; in particular,
$A_{nn_1} \ne \varnothing$, $A_{nn_2} \ne \varnothing $; if this is the case,  then $A_{n_1n_2} =\varnothing$.
%and $ A_{n_1n_2n_3}$ is basic for distinct $n_1,n_2,n_3$;
\item[$(4)$] $A_{n_1n_2n_3n_4} = \varnothing$ for distinct $n_1,n_2,n_3,n_4$;
\item[$(5)$] for each $n \in \bN$, the sequence $(A_0,\ldots,A_{n})$ is a Cartan $n+1$-string.
\end{itemize}
Then, we define such a sequence $\mathcal{A}$  to be
a {\em Cartan sequence in $\Omega$.}

 Let $\mathcal {U}=\{U_{\lambda}\}_{ \lambda \in \Lambda}$ be an % (locally finite)
axially symmetric open covering of $\Omega$,
let $Z \subset \Omega$ be a discrete set of points or spheres $S=\{x+Iy\ :\ I\in\mathbb{S} \}$ and $\mathcal{A}$  a
 Cartan sequence in $\Omega$. If each $A_n$ is contained in  an open set $U_{\lambda} \in \mathcal U$ and $A_{n_1n_2} \cap Z = \varnothing$ for $n_1 \ne n_2$, then we say that $\mathcal{A}$ is a {\em Cartan sequence subordinated to the pair $({\mathcal U},Z)$.} % If the covering $\mathcal U$ is not given, we say that the Cartan sequence is subordinated to $Z$ if it is {\em subordinated to the pair $(\{\Omega\},Z)$.}
 If, in addition, % $\mathcal{A}=\{A_n, n \in \bN_0\}$ is a Cartan sequence  subordinated to  $({\mathcal U},Z)$ and
 the  sets in the sequence ${\mathcal{B}}:=\{B_n:=\mathring{A}_n, n \in \bN_0\}$ form a covering of $\Omega,$ then we call the sequence $\mathcal{B}$ a {\em Cartan covering } subordinated to  $({\mathcal U},Z)$.\\
\end{definition}
Notice that, being basic, a set $\mathring{A}_n \in \mathcal {B}$ intersects at most one connected component of  $\Omega \cap \R$.\\

For the reader familiar with cohomology groups with values in a sheaf, let us briefly explain the reasons for chosing coverings with the listed properties.
As mentioned, the set of real points in a domain $\Omega \subset \H$ plays a different role than nonreal quaternions. For example, we have seen in Proposition \ref{ConnGroup} that the group of  nonvanishing slice--preserving functions on a basic slice domain, which is the interior of a closed basic set, has two components. In contrast, on a basic product domain, it has one component. In particular, there is no quaternionic logarithm of $-1$ in the class of slice--preserving functions on a basic slice domain (\cite{AdF2, GPV}).  Condition $(3)$ says that the covering has order $2$ when restricted to the reals  and that the nonempty intersection of  two  sets  that intersect the real axis is connected and also intersects the real axis  but all such double intersections are disjoint;  when dealing with cocycles of slice--preserving  nonvanishing  functions, this enables us to choose representatives of a cocycle in the connected component of $1$, where slice--preserving logarithm exists (\cite{AdF2, GPV}). For similar reasons related to the existence of logarithm, we require in $(2)$ that  $A_{n_1n_2}, A_{n_1n_2n_3}$ for distinct $n_1,n_2,n_3$ are closed basic sets and  the intersections $(A_1\cup \ldots \cup A_{n-1}) \cap A_n$ consist of a finite number of disjoint closed basic sets;  namely on such sets one has the possibility of finding slice--regular logarithmic functions.  Because the topology of an axially symmetric domain $\Omega$ is determined by the topology of $\Omega_I$, %and $H^q(\Omega_I, \mathbb{Z}) = 0 $ for $q\geq 2,$
we require that the covering reflects this fact: conditions $(2)$ and $(4)$ imply that  the nerve of the covering is planar. The requirement $(5)$, among other things,
says that besides $(A_0,\ldots, A_{n-1})$ also $(A_1 \cap A_n,\ldots A_{n-1} \cap A_n)$ is a Cartan string and this allows us to work with higher cohomology groups. Sometimes  certain subclasses of slice--regular functions intrinsically determine a discrete set $Z$ of points and spheres,  which have to be avoided,  therefore, we require that $Z \cap A_{n_1n_2} = \varnothing$.\\

%%\begin{definition}\label{simple_covering} A {\em Cartan covering} is a basic covering $\mathcal B = \{B_n, n \in \bN\},$ which is also a Cartan sequence with the given enumeration.
%%\end{definition}

The assumption of axial symmetry for the sets considered enables us to search
for such coverings by restricting the problem to (any) slice $\C_I$.
Recall that if $B \subset \H$ is a basic set, then for each $I \subset
\bS,$ the set $B_I$ is simply connected if $B$ is a slice domain. If
$B$ is a product domain then $B \cap \R = \varnothing$ and $B_I =
B_{I}^+ \cup B_{I}^-$ is a  union of two disjoint simply connected open
sets. In particular, the sets $B_I$ are always symmetric in the sense
of Definition \ref{sym}.

The idea is to define a fine enough symmetric grid $\Gamma$ in
$\Omega_I,$ $\Gamma \cap {(\widetilde Z)}_I = \varnothing,$
such that the
regions cut by the grid define a tiling $\mathcal{T}$ of the set
$\Omega_I,$ subordinated to $\mathcal {U}_I,$ with the tiles being
closed topological discs (or symmetric pairs of such) with piecewise
smooth boundaries satisfying $T_{l_1l_2l_3} \cap \R = \varnothing$ and $
T_{l_1l_2l_3l_4} = \varnothing.$ The tiles, listed in the correct order,
define a sequence of compact sets $\{T_l\}_{l\in \bN}$ and the
symmetrizations (in $\H$) of their suitable  open tubular neighbourhoods
with piecewise smooth boundaries give the desired Cartan covering
subordinated to ${\mathcal U}.$

\begin{remark}
The requirement that the grid misses the discrete set ${(\widetilde Z)}_I$ of points and spheres is
easy  to  achieve  by  locally   perturbing  the  grid.  Small  enough symmetric
perturbations do  not destroy other properties.  Therefore, it suffices
to  construct   a  grid  such   that  the  tiles  fulfill   all  other
requirements, i.e. from now on $Z = \varnothing.$
\end{remark}

Let us first define precisely what a symmetric tiling  of a symmetric set in $\C_I$ is.

\begin{definition}[Tiling]\label{def_tiling}
Let $K \subset \C$ be a compact symmetric set with a piecewise smooth
boundary and $\mathcal{U}$ a symmetric open covering of $K$.  {\it A
  symmetric tiling $\mathcal{T}$ of $K$ subordinated to $\mathcal{U}$}
is a finite sequence $\mathcal{T} = (T_0,T_1,\ldots, T_m)$ of
symmetric closed sets with piecewise smooth boundaries and disjoint
interiors such that for each $T_l$ there exists $U \in \mathcal U$ and
$T_l \subset U$ and the following holds:
\begin{trivlist}{}{}
\item[](i) $\bigcup\limits_{l=0}^m T_l = K;$
\item[](ii) each $T_l$  and is either a closed
  topological disc (if it intersects $\R$) or $T_l = D \cup R(D)$ with $D \subset \C^+$
   a closed topological disc (if $T_l$ does not intersect $\R$);
  moreover, it intersects the union of the previous tiles only in
  boundary points, i.e. for each $ l =
  1,\ldots,m$, then we have $T_l \cap (T_0 \cup \ldots \cup T_{l-1}) =
  \partial T_l \cap \partial (T_0 \cup \ldots \cup T_{l-1})$ and this set is either empty or  a
  finite union of disjoint piecewise smooth closed arcs; the set $T_{l_1l_2} , l_1 \ne l_2,$ if not empty, is either  a symmetric closed arc or
  $T_{l_1l_2} = \gamma_1 \cup R(\gamma_1)$ with
  the arc $\gamma_1 \subset \C^+$, %\gamma_2,$ $\gamma_1 \subset \C_i^+$ and $\gamma_2 = R(\gamma_1)$.
\item[](iii) for $0 \leq l_1 < l_2 < l_3 \leq m$  each set of the form $T_{l_1l_2l_3} $  or $T_{l_1l_2}\cap \partial K $, if not empty, consists of a pair of symmetric nonreal
%at most finitely many
points which are called the {\em vertices} of the tiling and  $T_{l_1l_2l_3l_4} = \varnothing $ for $0 \leq l_1 < l_2 < l_3 < l_4 \leq m$;
\item[](iv) if $T_{l} \cap \R \ne \varnothing $, then  $\partial T_{l} \cap \R $ consists of two points;
  moreover, there exists at most two tiles $T_{l_1}, T_{l_2}$
 which  intersect the real axis and also $T_{l}$; in this case $T_{l_1 l_2} =
  \varnothing.$
\end{trivlist}

If $D \subset \C$ is a symmetric open set and $\mathcal {U}$ a
symmetric open covering of $D$, then a {\em symmetric tiling
  $\mathcal{T}$ of $D$ subordinated to
  $\mathcal {U}$} is an infinite sequence $\mathcal{T} = (T_0,T_1,
\ldots )$ such that for each $m,$ the sequence
%$\mathcal{T}^m:=*T_0,\ldots, T_m)$ is a symmetric tiling of the
 $\mathcal{T}^m:=(T_0,\ldots, T_m)$ is a symmetric tiling of the
compact symmetric set $T_0 \cup \ldots \cup T_m$, subordinated to
$\mathcal{U}$, and fulfills also the condition $\bigcup\limits_{l=0}^{\infty} T_l =
D.$

%If $K \subset \C_I^+$ then the {\em tiling $\mathcal{T}$} of $K$ is
%the symmetric tiling of $K \cup \refl{K}$ restricted to $K$ (and
%analogously for $D \subset \C_I^+$).
If $K \subset \C_I^+$ then the {\em tiling $\mathcal{T^+}$} of $K$ is
defined in the same manner as the symmetric set's tiling but with the requirement for the symmetry dropped (and analogously for $D \subset \C_I^+$).

 Given $\delta > 0,$ a tiling is a {\em $\delta$-tiling}, if the diameters of the connected components of the tiles are less than $\delta.$
\end{definition}

\begin{remark} The tiles which intersect the real axis form
  a `chain' of closed topological discs which covers the set $D \cap \R.$ Their intersections with the real axis are
   bounded closed intervals.
\end{remark}

\begin{remark}\label{EasyTiling} If $K \subset \C_I^+,$ and we have a tiling ${\mathcal{T}}^+ = (T_0,\ldots,T_m)$ of $K$ fulfilling all
  the conditions except the requirement that the tiles are symmetric,
  then $\mathcal{T}:= (T_0 \cup \refl{T_0},\ldots,T_m\cup
  \refl{T_m})$ is a symmetric tiling of $K \cup \refl{K}.$
\end{remark}

In general, the union of tilings of two sets with disjoint interiors is
not a tiling of the union of these sets.
\begin{definition} Let $A, B$ be two compact sets and ${\mathcal T}_A = (T_{A,0},\ldots, T_{A,j})$, ${\mathcal T}_B=(T_{B,0},\ldots, T_{B,k})$ their tilings, then
\[{\mathcal T}_A \underrightarrow{\cup} {\mathcal T}_B:=
(T_{A,0},\ldots, T_{A,j}, T_{B,0},\ldots, T_{B,k}).\] If this is tiling of $A \cup B$ then
we say that we have {\em extended the tiling} from the set $A$ to the set $A
\cup B$ and that the tiling ${\mathcal T}_A \underrightarrow{\cup} {\mathcal T}_B$ is {\em an extension} of the tiling ${\mathcal T}_A$ with the tiling ${\mathcal T}_B.$
\end{definition}

\noindent Proposition \ref{tiling1} represents  a part of Theorem \ref{BasicCovering}.
\begin{proposition}\label{tiling1} Let $\Omega \subset \H$ be an axially symmetric domain, together with $\mathcal{U} = \{U_n\}_{ n \in \bN}$, a locally finite axially symmetric
  covering of $\Omega$ and $\mathcal{T}$,  a symmetric tiling of $\Omega_I$ subordinated to $\mathcal{U}_I$ for some $I \in \mathbb S$.  Then $\mathcal T$ generates a Cartan covering $\mathcal{B}$ of $\Omega$ subordinated to ${\mathcal U}$.
\end{proposition}

\begin{proof}
For each $l \in \bN_0 $, let $D_l \Subset \Omega_I $ be an  axially symmetric  tubular neighbourhood of $T_l$  %with piecewise smooth boundary and
such that
\begin{itemize}
\item[$(1)$] if $T_l \subset U_{j,I}$ then $T_l \subset D_l \subset \overline{D_l} \subset
  U_{j,I},$
\item[(2)] if nonempty, the intersection $D_{l} \cap \R $ is
  connected and $\partial D_{l} \cap \R $ consists of two points;
  moreover, there exists at most two sets $D_{l_1}, D_{l_2}$
 which  intersect the real axis and  $D_{l}$ and $D_{l_1 l_2 } =
  \varnothing$;
\item[(3)] for $0 \leq l_1 < l_2 < l_3 \leq l$
 each set of the form  $D_{l_1l_2l_3} $ or $D_{l_1l_2}$, if not empty, is either a symmetric topological disc, if it intersects the real axis, or, if it does not,  it is equal to $U \cup R(U)$, where $U$ is a topological disc contained in the upper half-plane.
Moreover,   $D_{l_1l_2l_3l_4} = \varnothing $ for $0 \leq l_1 < l_2 < l_3 < l_4 \leq l$;
\item[$(4)$] the intersections $D_l \cap (D_0 \cup \ldots \cup
  D_{l-1})$ are  tubular neighbourhoods of the arcs $T_l \cap (T_0 \cup
  \ldots \cup T_{l-1})$  i.e. finite unions of open topological discs, such that also
  their closures are disjoint and   the sets $D_l$ and $(D_0 \cup \ldots \cup
  D_{l-1})$ enjoy the separation property (Definition \ref{CartanPair}).
\end{itemize}
\begin{figure}[h]
$\begin{array}{cc}
 {\includegraphics[width=0.4\textwidth]{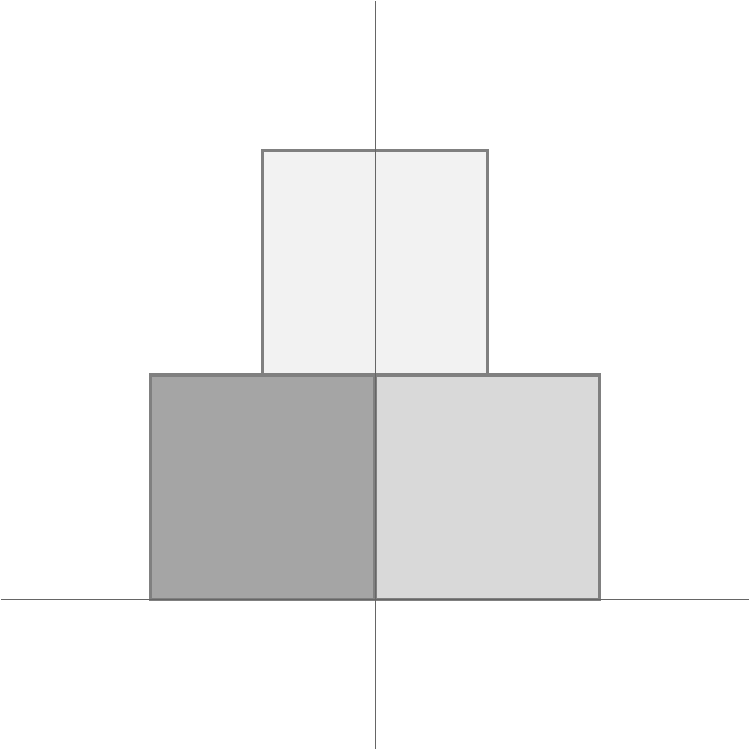}}&
{\includegraphics[width=0.4\textwidth]{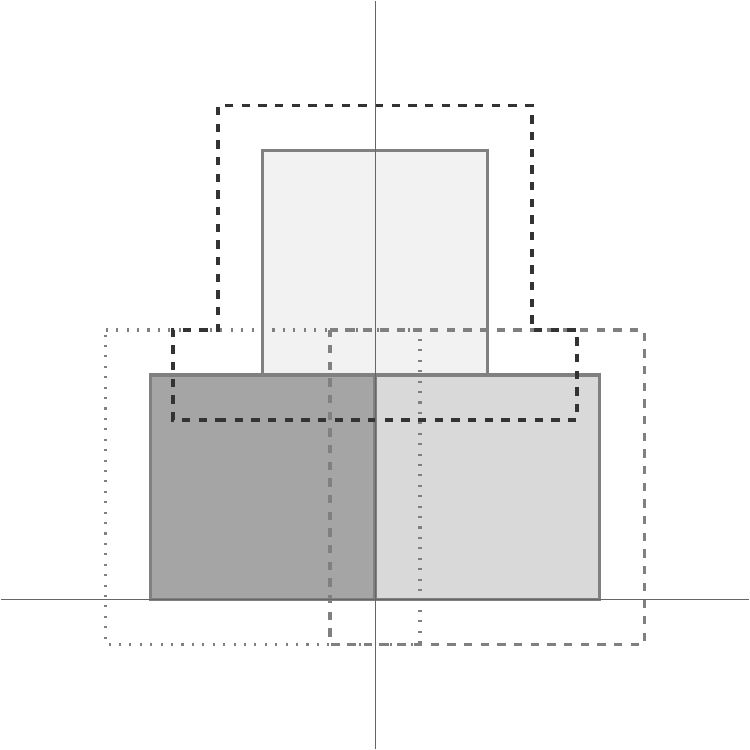}}
\end{array}$
\caption{Tiles (left) and  tiles with neighbourhoods (right).}
 \label{TilNbhd}
 \end{figure}

As depicted in Figure \ref{TilNbhd}, the neighbourhoods are obtained by enlarging the tiles's interiors. The separation property means that their closures  share a piece of boundary. This property is insured if the tiles are in the same geometric position as the left and the right tile in  Figure \ref{TilNbhd}. If we place a tile on top of them, we have to enlarge the neighbourhood near the  boundary the top tile shares with the lower ones to achieve the separation property (black dashed line in Figure \ref{TilNbhd}).

These properties ensure that if $T_l$ is a closed topological disc with a piecewise smooth boundary,
then $D_l$ is an open disc, and if $T_l$ is a union of two closed disjoint
topological discs, then the set $D_l$ is a union of two open disjoint
topological discs. Without loss of generality we assume that the
closures of these two discs are also disjoint (else we shrink them a
little).  Denote by $B_l:= \tilde{D}_l$ the axial symmetrization of
$B_l$. Then $\mathcal{B}:=\{B_l\}_{l \in \bN_0}$ is a Cartan
covering of $\Omega.$
\end{proof}

From now, on we restrict our considerations to constructing the symmetric tilings of $\Omega_I.$

\begin{example}\label{TilingOfASquare}
As the first model example, we present a tiling of the square $[-1,1] \times
[-1,1] \subset \R^2 \cong \C_I$ which also serves as a model case. The
tiling is obtained from a symmetric grid which consists of horizontal
and vertical lines chosen in the manner presented in Figure
\ref{BricksS} (a).  By choosing a finer division in the coordinate
directions, the resulting tiles can be as fine as we wish.  In
addition, if a finite number of points is given on the boundary of the
square, we can choose the horizontal and vertical lines so that the
intersection points of the boundary of the square and vertical and
horizontal segments do not contain any of the given
points.  The tiles are listed in such an order, that the tiles on the
real axis come first and then the tiles which consist of pairs of
symmetric regions are added so that the distance from the real axis is
increasing.
\begin{figure}[h]
$\begin{array}{cc}
 {\includegraphics[width=0.4\textwidth ]{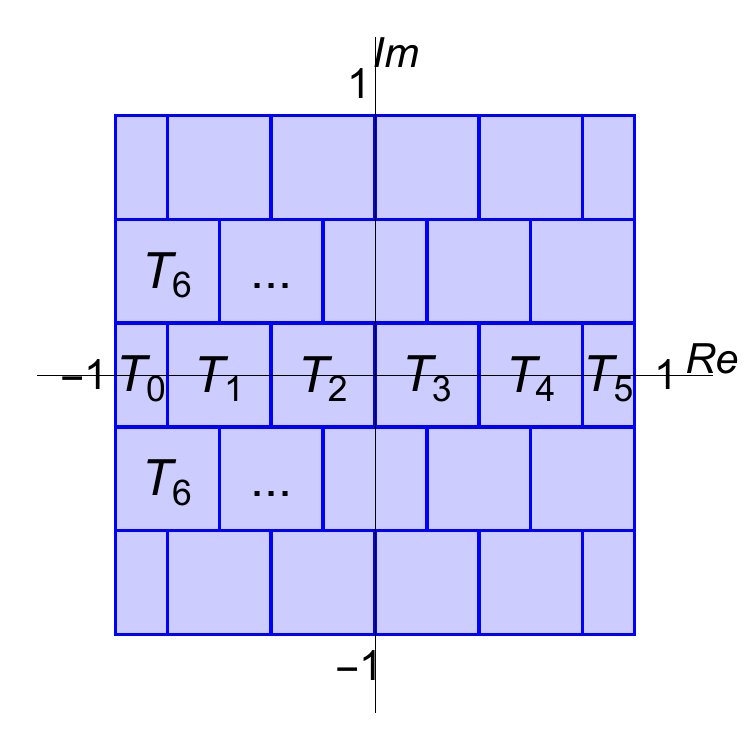}}&
{\includegraphics[width=0.4\textwidth]{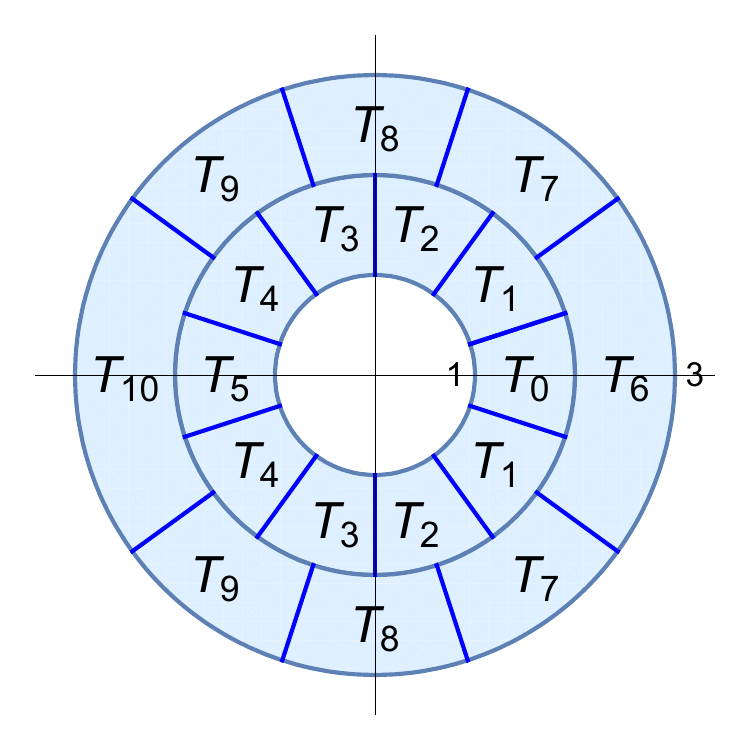}}\\
(a)&(b)
\end{array}$
\caption{(a) Symmetric tilings of the square  and (b) the annulus $\overline{A(0;1,3)}$ with $c_0 = 1,c_1 =2$ and $c_2 = 3$.}
 \label{BricksS}
 \end{figure}
\end{example}
\begin{example} \label{TilingOfAnnuli} As another example,  consider  the closed annulus $\overline{A(0;1,3)}$ (Figure \ref{BricksS} (b)).  In a slice $\C_I$, the grid can
  be defined by using polar coordinates.
Given an axially symmetric open covering $\mathcal{U}$, the
intersection with $\C_I$ defines a symmetric open covering of $\C_I.$

 To tile the closed annulus $\overline{A(0;1,3)},$ divide $[1,3]$ to $1 =
 c_0 < c_1 < \ldots < c_m = 3$ and $[0,\pi]$ to $0 < \vph_0<\vph_1<
 \ldots < \vph_{2k+1} < \pi.$ Cut the annuli
 $\overline{A(0,c_{2i},c_{2i+1})}$ with rays $\vph = \pm \vph_{2j}$ and
 the annuli $\overline{A(0;c_{2i+1},c_{2i+2})}$ with rays $\vph = \pm
 \vph_{2j+1}.$
 If both partitions are fine enough, then the grid defined by
 the circles  $r = c_i$ and by the rays  $\vph = \pm \vph_{2j}$
 %% is fine enough, so
 is such
 that each tile defined by this grid is contained in a member of
 the covering. The same holds for the grid defined by the circles
 $r = c_i$ and by the rays  $\vph = \pm \vph_{2j+1}.$
%\begin{figure}[h]
%{\includegraphics[width=0.4\textwidth]{TAnnN.pdf}}%{TilingsFirstExample.eps}}
%Bojan,{Symmetric tiling for the annulus $\overline{}(A(0;1,3))$ with $c_0 = 1,c_1 =2$ and $c_2 = 3.$}\label{TFE}
%\end{figure}

The tiles have to be listed in the correct order in the sense that for
each $i,$ all tiles in $\overline{A(0,c_i, c_{i+1})}$ are listed before
those in $\overline{A(0,c_{i+1}, c_{i + 2})}.$ It is obvious that in this
manner, a newly added tile intersects the previously added tiles in a
union of  smooth arcs.
\end{example}

\begin{example}\label{TilingOfPairOfSquares}
As the third model example (Figure \ref{BricksU}), we present a tiling of the union of squares $[-1,1] \times [-3,-1] \cup [-1,1] \times [1,3].$
\begin{figure}[h]\centering
$\begin{array}{ccccc}
{\includegraphics[height=0.26\textheight]{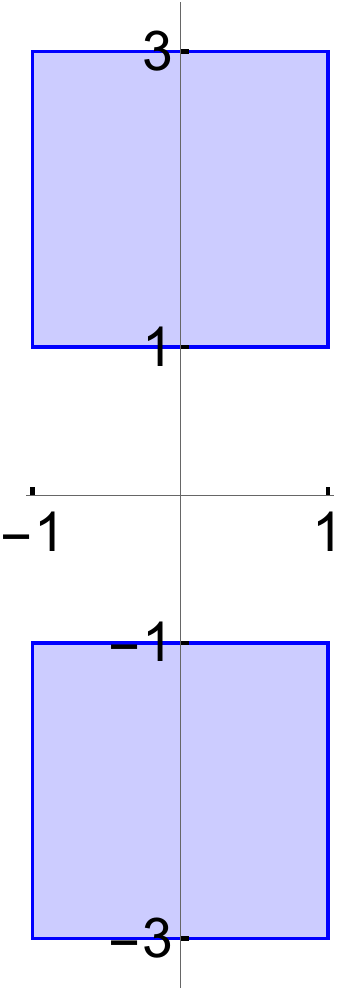}}& &&&
{\includegraphics[width=0.4\textwidth]{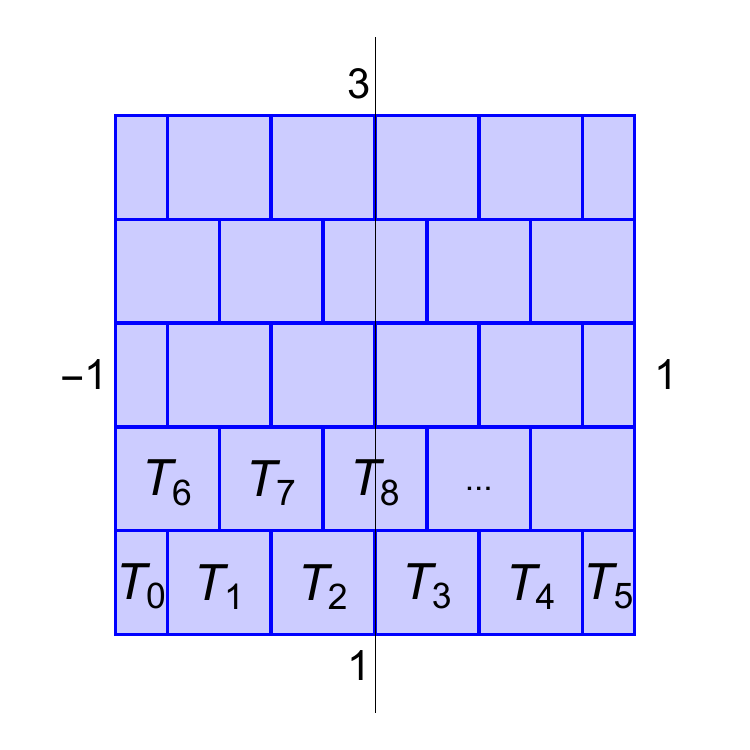}}\\
(a)&&&&(b)
\end{array}$
\caption{(a) Pair of squares and (b) tiling of the model square $Q=[-1,1] \times [1,3]\qquad\qquad$ %\phantom{aaaaaaaaaaaa}
}\label{BricksU}
\end{figure}

By Remark \ref{EasyTiling}, it suffices to construct a nonsymmetric tiling of $[-1,1] \times [1,3]$  and then extend it to
$[-1,1] \times [-3,-1]$ by reflection.
The tiling is obtained from a grid which consists of horizontal and vertical lines chosen in the manner presented on Figure \ref{BricksU}.
The resulting tiles can be as fine as we wish by choosing a finer division in the coordinate directions.
As in the previous model example, the tiles with a smaller distance from the real axis are listed first.

Notice that the tiling of a square in Example \ref{TilingOfASquare}
can be obtained by tiling the upper rectangle $[0,1] \times [-1,1]$
following the scheme in this example and then symmetrize the tiles.

%% e}[h]
%% ar}{ccccc}
%% phics[height=0.26\textheight]{SimBricks.pdf}}& &&&
%% phics[width=0.4\textwidth]{BricksU.pdf}}
%% }
%% r of squares and tiling of the square $[-1,1] \times [1,3]$ \phantom{aaaaaaaaaaaa}}\label{BricksU}
%%
\end{example}

\subsection{Symmetric tilings for the general case}

The main result of this section, Theorem  \ref{MainTheoremTiling}, whose proof will be given after providing some extra tools, is a key ingredient for the proof of  Theorem \ref{BasicCovering}.
\begin{theorem} \label{MainTheoremTiling} Let $\Omega$ be an axially symmetric  domain and 
  ${\mathcal U}_I=\{U_{I,n}\}_{n \in \bN}$ a symmetric locally finite open
  covering of $\Omega_I$.  Then, there exists a symmetric tiling
  ${\mathcal T}$ of $\Omega_I$ subordinated to $\mathcal{U}_I.$
\end{theorem}

%%%\begin{corollary} \label{BasicCovering} Let ${\mathcal U}$ be a locally finite axially symmetric
%%%  open covering of an axially symmetric domain $\Omega$ and $Z \subset
%%%  \Omega$ a discrete set of points or spheres. Then $\Omega $ admits a
%%%  Cartan covering $\mathcal B$ %n axially symmetric Cartan covering
%%%  subordinated to $({\mathcal U},Z).$
%%%
%%%  Even more, there exists a Cartan covering $\mathcal B = \{B_n\}_{ n \in \bN_0}$
%%%  and a sequence $\{\ve_n\}_{n \in \bN_0}$ so that also the coverings
%%%  \[ {\mathcal B}^j:= \{ B_n + B(0,\ve_n/2^j)\}_{ n \in \bN_0} \]
%%%  are simple for all $j \in \bN_0.$
%%%\end{corollary}
%%\marginpar{JASNA: family of subordinated coverings; DONE }
\begin{proof}[Proof of Theorem \ref{BasicCovering}.] Theorem \ref{BasicCovering} follows directly from Theorem  \ref{MainTheoremTiling} and Proposition \ref{tiling1}.
 It is obvious from the construction in the proof of Proposition \ref{tiling1} that a sequence $\{\ve_n> 0\}_{n \in \bN_0}$ exists so that
the family $\mathcal B^t$ defined in Theorem \ref{BasicCovering} has all the desired properties.
\end{proof}

\begin{remark}
In complex analysis,  a covering of this sort appears
when considering a Morse function.  An extensive explanation of this approach can be found in \cite{FF}, subsection 3.9.
In particular, in  one complex variable, one could
use the regular level sets and flows of the gradient vector field
of the Morse exhaustion function of a domain (and its small perturbations) to obtain a sufficiently fine
grid in the domain.

Also, in several variables, the approach relies upon the Morse function, but the construction is different and uses
to the so--called ``bump method'' introduced
by Henkin-Leiterer in (\cite{hl}).

In general, for an axially symmetric slice domain $\Omega\subset\H$ a
symmetric Morse function $\Omega_I \ra \R$ may not exist, since in the
construction of a Morse function of a set, one has to use Sard's
theorem.  Even if it existed, controlling the number of sets intersecting for the covering restricted to the real axis is difficult.
%However, a strictly subharmonic exhaustion function
%$m:\Omega_I \ra \R$ always exists, and hence one would try to define a
%grid by regular level sets of $m$ and flows of the vector field $\grad
%m$ in the upper closed leaf and extend it to the whole domain by
%reflection.  Unfortunately, in this case, it is difficult to control
%the behaviour of the level sets and the gradient vector field near the
%real axis, since we do not have any control of the type of the
%critical points.
%For example, we can not use the function $m(z) =|z|^2 + |z|^{-2}$ for $\C_I \setminus \{0\},$ since the minimum is obtained on a circle.
%The regular level sets of a symmetric strictly subharmonic function
%intersect the real axis transversely and if a sublevel set touches
%the real axis from above, i.e. the in the open leaf $\C_I^+$ the
%gradient has a direction $(0,-1)$ and and $\C_I^-$ the direction
%$(0,1),$ then the point is a saddle point and hence it defines a
%symmetric plus tile.

Another possibility would be to approximate $m$ by a Morse function and construct the covering by reflecting its level sets and integral curves of the gradient vector field in the upper closed leaf.
%The approximated function may have
%a saddle point at the real
%axis and, moreover,
Unfortunately, if there were a degenerate critical point on the real axis, a regular level set of the approximated function  might intersect the real axis many times. Hence, the reflection of the sublevelset creates a hole that cannot be covered without creating a non-simply connected intersection.

\end{remark}

Before going to the proof we will define some useful notions.
 The following definition explains how we join the given family of smooth discs in a necklace.

\begin{definition}[$\mathcal D$-necklace] \label{necklaces} Let $D$ be a closed  topological disc with a piecewise smooth boundary  and let $\mathcal D :=\{D_i \subset D, i =1,\ldots, {k}  \}$ be a family of  closed  disjoint discs in $\mathring{D}$ with smooth boundaries.
Let $l_i:[0,1] \ra D, i = 1,\ldots,k$  be smooth disjoint arcs so that
\begin{itemize}
\item $\cup_{i = 1}^{k_1} l_i((0,1)) \cap D_i = \varnothing$   for all $i,$
\item $l_1(0) \in \partial D, l_1(1) \in \partial D_1$ and $l_i(0) \in \partial D_{i-1},$ $l_i(1) \in \partial D_{i},$ $ i = 2,\ldots, k,$
\item the intersections of the arcs with the boundaries of the discs are perpendicular.
\end{itemize}
Then the sequence $\{l_1,  D_1, l_2,  D_2,\ldots, l_k, D_k\}$ is called a {\em $\mathcal D$-necklace}. If the discs are all contained in $\C^+$, then we require that also the arcs are in $\C^+.$
If $D$ and the discs $D_i$ are symmetric and the arcs $l_i$ are segments on the real axis, the necklace is called a {\em symmetric $\mathcal D$-necklace}. If, in addition, a segment $l_{k+1}$ on the real axis joining $\partial D_k$ and $\partial D $ is added, then
$\{l_1,  D_1, l_2,  D_2,\ldots, l_k, D_k, l_{k+1}\}$ is called a {\em complete $\mathcal D$-necklace}. A complete necklace is {\em trivial } if $\mathcal D$ is empty and the arc is $l:= \mathrm{id}|_{D \cap \R}. $
\end{definition}

With the next definition, we explain how to define a tiling in a neighbourhood of a necklace.
%\subsection{$\delta$-tiling of a $\mathcal D$-necklace}
\begin{definition}
A {\em  $\delta$-tiling of a (symmetric) $\mathcal D$-necklace} $\{l_1,  D_1, l_2,  D_2,\ldots, l_k, D_k\}$ in a closed topological disc $D$ is defined to be any tiling $\mathcal T = (T_0,\ldots, T_n)$ of $T:=T_0 \cup \ldots \cup T_n \subset D$ so that for any $j$ the connected components of $T_j$ have diameter at most $\delta$,  the sets $l_j^*,\partial D_j \subset T$ and for any tile $T_l$ we have $T_l \cap \partial D_j = T_l \cap D_j$  (i.e. the tiles are attached to $\partial D_j$  from the outside (see Figure \ref{DBnecklace})). Moreover, we require that none of the sets  $\partial D_j$ is contained in the union of only two tiles.
\end{definition}

In the sequel, we will construct a specific tiling of a necklace $\{l_1, D_1\}$, where tiles are listed in the precise order, following the order of arcs and discs given by the necklace and their orientations.
%Assume that $\mathcal D =\{D_1\}$ and $\{l_1, D_1\}$ the $\mathcal D$-necklace.
Recall that if the necklace is symmetric (i.e. $D$ and $D_1$ are both symmetric), the arc is a segment on the real axis.
If  $D_1$ is in the upper half-plane, also the arc is in the upper half-plane.

Let $\lambda_1 : [0,1] \ra \partial D_1$ be a parametrization of positively oriented $\partial D_1$  and
let $n: \partial D \ra \C^*$ be its  unitary outer normal. Choose $0 = t_0 < t_1 < \ldots < t_j = 1$ for some $j \geq 3$ so that $ \lambda_1(t_i)\ne l_1(1)$,  $\diam \lambda_1([t_i,t_{i+1}]) < \delta/4$  and $\diam l_1([t_i,t_{i+1}])< \delta/4$  for $ i = 0,\ldots j-1$.  Moreover, if $\partial D_1$ is symmetric, we choose $t_i$ so that the points $d_i = \lambda_1(t_i)$ are  symmetric and nonreal. Let $r > 0$ and $\phi_{l_1}, \phi_{\lambda_1}$  be given by Lemma \ref{difeo} and assume that $\delta< r$ so that Lemma \ref{difeo} applies and, moreover, that $\diam(\phi_{\lambda_1}(\lambda_1([t_i,t_{i+1}]) \times [0,\delta/4]))< \delta$ and $\diam(\phi_{l_1}(l_1([t_i,t_{i+1}]) \times [0,\delta/4]))< \delta$ for $ i = 0,\ldots j-1$.

 {\em Tiles attached to $\partial D_1$.} Define the sets by $S_i:=\phi_{\lambda_1}(\lambda_1([t_{i}, t_{i+1}]) \times [0,\delta/4]),$ $i = 0,\ldots j-1.$  Without loss of generality we may assume that  $S_0$ is the only one of the sets  $S_i$ which intersects the arc $l_1.$ Set  $\mathcal T_{\lambda_1}:= (T_0:=S_0,\ldots, T_{j}:=S_j).$

In the case of a symmetric necklace, we define the tile $T_0$ in the same manner, and the rest of the tiles are unions of pairs of symmetric sets
$S_i$ following the orientation of the part of the circle in the upper half-plane, as in Figure \ref{BricksS}(b).

{\em Tiles covering $l_1$.} Choose $\delta_1 < \delta/4$ and
set $T''_i:= \phi_{l_1}(l_1([t_{i}, t_{i+1}]) \times [0,\delta_1])$ for $i = 0,\ldots j-1.$
Let $\delta_1$ be so small, that the tiles $T_i''$ do not intersect the tiles $T_1,\ldots, T_j.$
Define $T'_i:=  \overline{T_i''\setminus T_0}$, discard the empty tiles and list the remaining tiles in the tiling
$\mathcal T_{l_1}:=(T'_{0},\ldots, T' _{j'})$ following the orientation of the arc $l_1$. By choosing $\delta_1$ small enough, we achieve that the tiles $T'_{k}, k = 0,\ldots,j'$  do not contain  points
$\phi_{\lambda_1}(\lambda_1(t_i) \times \{\delta/4\})$ for $ i = 0,1$,  and so
$\mathcal T_{l_1} \underrightarrow{\cup} \mathcal T_{\lambda_1}$ is a $\delta$-tiling of the union of all of its  tiles. %The construction implies that the diameter of connected components of the tiles does not exceed $\delta.$
Any such a tiling is called  an {\em ordered $\delta$-tiling of the $\mathcal D$-necklace. } If the necklace is symmetric, then the constructed tiling is also symmetric and called a { \em symmetric ordered $\delta$-tiling of the $\mathcal D$-necklace. }

The construction  of the tiling of a necklace $\{l_1,D_1,\ldots,l_k, D_k\}$ or a $\mathcal D$-complete necklace is analogous, taking into account the ordering of the necklace. If all the discs are in the upper half-plane, so can be the necklaces and the tilings.

\begin{remark}
If, in addition, a finite set of points on $\partial D \cup (\cup_{D_j \in \mathcal D}\partial D_j)$ is given, the tiling can be chosen in such a way that the vertices of the tiling avoid the given points.
\end{remark}
\begin{remark}\label{DB} If we have two families of disjoint discs $\mathcal D$ and $\mathcal B$ in $D,$ it is straightforward that we can construct disjoint
$\mathcal D$- and $\mathcal B$-necklaces with disjoint tilings.
\end{remark}
An example of such a (symmetric) tiling with two families is presented in Figure \ref{DBnecklace}, $\mathcal D = \{D_1\},$ $\mathcal B = \{B_1, B_2\}$, where $D_1$ is a white disc  and $B_1, B_2$ are blue discs.

\begin{figure}[h!]
  \centering
\begin{center}
%{\includegraphics[width=0.4\textwidth]{LRnecklace.eps}}
{\includegraphics[width=0.4\textwidth]{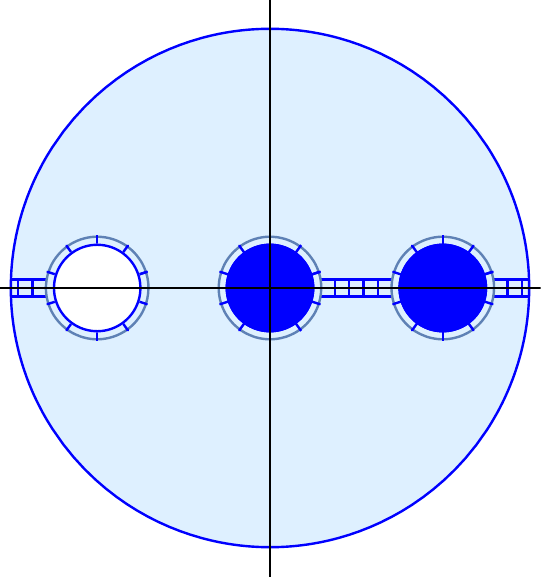}}
  \caption{Symmetric $\mathcal D$-and $\mathcal B$-necklaces and their tilings.}\label{DBnecklace}
\end{center}
\end{figure}

\begin{remark}\label{NeckTil}
  It follows from the construction above that given an arbitrary open covering $\mathcal U = \{U_{\lambda}\}_{ \lambda \in \Lambda}$ of a smooth closed disc $D\subset\C,$ one can choose  $\delta > 0$ so that the constructed ordered $\delta$-tiling of any  necklace in $D$ is subordinated to $\mathcal U$.  If, in addition, $D$, the open covering $\mathcal U$ and the necklace are all symmetric,  the ordered $\delta$-tiling  of $D$ can be constructed symmetric.
\end{remark}

\begin{proposition}\label{IndStep} Let $D$ be a closed disc with a piecewise smooth boundary,   $\mathcal D = \{D_i\subset D , i =1,\ldots, {k_1} \}$
  and $\mathcal B= \{B_j\subset D ,j = 1,\ldots {k_2} \}$
  families of smooth disjoint closed discs, which can be also empty, i.e. $k_1 = 0$ or $k_2 = 0.$ Set $\Delta:={D} \setminus \mathrm{ int }(D_1 \cup \ldots \cup D_{k_1} \cup B_1
  \cup \ldots \cup B_{k_2})$ and let $V \subset \partial \Delta$
  be a discrete set of points containing all nonsmooth points of $\partial D$.

Let $\delta > 0$ be given.
Then there exists a $\delta$-tiling $\mathcal T_{\Delta}$ of $\Delta,$ so that %the (connected components
%of) the tiles have a diameter less than $\delta$ and
the vertices of the
tiling are disjoint from $V.$ More precisely,
\[\mathcal T_{\Delta}=\mathcal T_{\mathcal B} \underrightarrow{\cup} \mathcal T \underrightarrow{\cup}\mathcal T_{\mathcal D},\]
where $\mathcal T_B$ is a $\delta$-tiling of the $\mathcal B$-necklace, $\mathcal T_D$
is a $\delta$-tiling of the $\mathcal D$-necklace and $\mathcal T$ is a tiling of the closure of the set $\Delta$ with tiles from $\mathcal T_{\mathcal B}$ and $\mathcal T_{\mathcal D}$ removed.

\end{proposition}

\begin{proof} Let $\mathcal T_{\mathcal B}$ be an ordered $\delta$-tiling of the $\mathcal B$-necklace, $\mathcal T_{\mathcal D}$
an ordered $\delta$-tiling of the $\mathcal D$-necklace with the vertices of both tilings disjoint from $V.$
Define $T_B:= \cup_{T \in \mathcal T_B} T,$ $T_D:= \cup_{T \in \mathcal T_D} T,$ and let
$K:= \overline{\Delta \setminus ({T_B}\cup {T_D})}.$ The set $K$ needs to be tiled.

Assume that $k_1, k_2 > 0.$ The set $K$ is a topological disc with a
piecewise smooth boundary and can be mapped by a map, smooth up to the
boundary (except at finitely many points, where it is only continuous, compare Theorem 3.9 in \cite{Pom}) to the model square $Q = [-1,1]\times[1,3]$ of
Example \ref{TilingOfPairOfSquares},  Figure \ref{BricksU} (b),  in such a
way that $\partial K \cap T_B$ is mapped
to $ [-1,1]\times\{1\} $ and  $\partial K \cap T_D$  is mapped to $ [-1,1]\times\{3\}.$
As  $\delta$-tiling $\mathcal T$ of $K$ consider the one induced by a fine enough tiling of the model square $Q$.
We may assume  that the images of vertices of $\mathcal T_B$ and $\mathcal T_D$  and the images of set $V$ are not the
endpoints of horizontal and vertical segments of the grid of $Q.$
%provided by the Example \ref{TilingOfPairOfSquares}.
Then $\mathcal T_B \underrightarrow{\cup} \mathcal T \underrightarrow{\cup}\mathcal T_D$ is the desired tiling of $\Delta.$

If $k_1 = 0$ (or $k_2 = 0$), we apply the same  argument without the
requirements that the $\partial K \cap T_B$ (or $\partial K \cap T_D$) is mapped
to the edge $ [-1,1]\times \{1\}$ (or $ [-1,1] \times \{3\}$).

%If $k_1 = 0$ and $k_2 > 0$ we use the previous case without the
%requirement that the  $\partial K \cap T_D$ is mapped
%to the edge $ [-1,1] \times \{3\}.$

%If $k_1 = k_2 = 0,$ then $K = \Delta = D$, and we use the previous case without the requirements on images of $\partial K \cap T_B$ and $\partial K \cap T_D.$
\end{proof}

%%\begin{remark}\label{NoHoles}
%%If $k_1 = 0,$ then we can require that only a closed arc $\gamma \subset (\partial K \setminus T_B)$ is mapped
%%to the edge $[-1,1] \times \{3\} .$
%%\end{remark}

\subsection{Proof of Theorem \ref{MainTheoremTiling}}

Since the problem for axially symmetric domains which are not
slice domains will be covered as a subproblem for slice domains, we
assume that $\Omega_I$ is a slice domain.

We proceed in two steps: by Proposition \ref{SymCov}, we first  exhaust $\Omega_I$ by
   symmetric compact sets $\{K_n\}_{n \in \bN_0}$ with smooth boundaries  and then
   proceed by induction. Once the tiling $\mathcal T_n$ of $K_n$   is defined, we extend it to the tiling of $K_{n+1}$ by tiling the set
              $K_{n+1} \setminus \mathring{K_n}$ in the following manner:
              we first extend the tiling of ${K}_n$  with  a tiling of $K_{n+1}\setminus \mathring{(K^{\bullet}_n)}$ (extension outwards), and then we extend the new tiling with a tiling of the set
              $K^{\bullet}_n \cap {K}_{n+1}$ (extension inwards).
%\marginpar{Add details about the inward(outward) tiling}

The initial tile is $T_0 = K_0$; recall that $K_0 \cap \R \ne \varnothing$. For the induction step, let
$\mathcal{T}_n$ be the symmetric tiling of $K_n$ which is already
defined.  Observe that because the covering is locally finite, there exists $\delta_{n+1} > 0$ such that for
each $z \in K_{n+1}$ the disc $B(z,\delta_{n+1})$ is contained in all
of the open sets of the covering $\mathcal{U}_I$ that intersect $B(z,\delta_{n+1})$. Fix such a
$\delta_{n+1}.$ To satisfy the condition that the tiling is
subordinated to the given covering, it suffices to construct tiles such
that their connected components have a diameter less than
$\delta_{n+1}.$

Denote the union of connected components of $K_n,$ which intersect the
real axis by $K_{n}',$ and the union of those which do not by $
K_{n}''.$ Hence $K_n = K_n' \cup K_n'',$ similarly $K_{n+1} = K_{n+1}'
\cup K_{n+1}'';$  analogously, for the filled components of $K_n,$ we
set ${K}^{\bullet}_n = {K}'^{\bullet}_n \cup
{K}''^{\bullet}_n,$ where ${K}'^{\bullet}_n$
(resp. ${K}''^{\bullet}_n$) denote the union of filled components
of $K_n$ that intersect (resp. do not intersect the real axis). Notice that $(K_n')^{\bullet} = (K^{\bullet}_n)',$
$(K_n'')^{\bullet} = (K^{\bullet}_n)''.$

Let $K_{n+1,\nu}$ be a connected component of $K_{n+1}.$ Put $K_{-1}:= K_n \cap K_{n+1,\nu}.$  We extend the
tiling to the set ${K_{n+1,\nu}}$ in  two
steps. First  we extend the tiling to ${K_{n+1,\nu}} \setminus \mathring{(K_{-1}^{\bullet})}$ (outwards) and then to
${K}_{-1}^{\bullet} \cap K_{n+1,\nu}$.
(inwards).

We also distinguish two cases: (1) if the set  $K_{n+1,\nu}$ is in the upper half-plane, we extend the tiling from $K_n$ to $K_{n+1,\nu}$ and then  extend it to  ${K_{n+1,\nu}} \cup \refl{K_{n+1,\nu}}$ by reflection over the real axis; (2) if $K_{n+1,\nu}$ intersects the real axis, it is symmetric and hence we have to extend the symmetric tiling of $K_n$ to a symmetric tiling $K_{n+1,\nu}.$ \\
%Notice that, since $K_{n+1,\nu}$ is connected, also the
%intersection $K_{n+1,\nu} \cap {K}_{n, \mu}$ is connected for
%all ${K}_{n, \mu} \subset K_{n+1,\nu}.$

\noindent Case 1: $K_{n+1,\nu} \subset K_{n+1}''.$  Then either  $K_{n+1,\nu} \subset \C_I^+$ or  $K_{n+1,\nu} \subset \C_I^-$ and without loss of generality we assume that  $K_{n+1,\nu} \subset \C_I^+.$  By Remark \ref{EasyTiling} it suffices to define the tiling of $K_{n+1,\nu}$ and then extend it to ${K_{n+1,\nu}} \cup \refl{K_{n+1,\nu}}$ by reflection.

Consider the set $K_{n+1,\nu} \cup {K}^{\bullet}_{-1}.$
It is a topological closed disc with finitely many holes  $D_1,\ldots D_{k_1}$
and contains (at most)  finitely many closed  discs $B_1,\ldots B_{k_2}$ which are all the connected components of ${K}^{\bullet}_{-1}.$  Put
$K:=K_{n+1,\nu}^{\bullet},$ $ \mathcal D = \{D_1,\ldots D_{k_1} \},$ $\mathcal B = \{B_1,\ldots B_{k_2} \}.$

{\em (i) Extending the tiling outwards.} %(we tile $K_{n+1,\nu} \setminus {K}^{\bullet}_{-1}$)}:
Proposition \ref{IndStep} with $D := K$ and families $ \mathcal B, \mathcal D$ of closed discs  yields the desired tiling so that the vertices of the new grid that are on the set $K_{-1}$ are different from the ones on $\partial K_{-1}$ induced by the tiling of $K_n.$

{\em (ii) For the  extension of the tiling inwards} % by adding the tiling of  the set $(K_{-1}^{\bullet} \cap K_{n+1, \nu})\setminus K_n,$
it suffices to explain the construction for one of the discs $B \in \mathcal B,$ since they are disjoint.

If $B$ is  a connected component of $K_n,$  it is already tiled, and there is nothing to do.
Hence, assume that $B$ is not a connected component of $K_n$ and therefore
 $B \cap K_n$ has at least one hole. Because $K_n$ is Runge in $K_{n+1}$,
the intersection $B \cap K_{n+1, \nu}$ also  has  at least one smaller hole in each hole of $B\cap K_n$.   Let $D_1',\ldots , D'_{k_1}$ be the holes of  $K_{n+1, \nu}$ contained in $B$.
Take a small disc $B(p,r) \Subset D_1'$ and reflect the set
$B \cap K_n$ across the circle $\partial B(p,r).$ The reflection transforms the problem of extending inwards to the problem of extending outwards, and by (i), we can extend the tiling. \\ %We conclude by reflecting the tiling provided in (i).\\

\noindent Case 2: $K_{n+1,\nu} \subset K_{n+1}'.$
Also, in this case, we proceed similarly.  First, we fill outwards and then inwards.

%If $K_{n+1,\nu}$ is a smooth disc and $K_{-1} = \varnothing$,% K_{n+1,\nu} \cap K_n = \varnothing,$
%we get the symmetric tiling  from Example \ref{TilingOfASquare} since there exists a homeomorphism of the square to $K_{n+1,\nu},$ which is smooth up to the boundary except at the corners of the square. %, \ref{TilingOfPairOfSquares}.
%If $K_{n+1,\nu}$ is a disc and $ K_{n+1,\nu} \cap K_n \ne \varnothing,$ then let $\mathcal B$ be the union of filled components of $K_n$ intersecting $K_{n+1, \nu}.$
%or Proposition \ref{IndStep} otherwise.

%Now assume that $K_{n+1,\nu}$ is not a disc or $ K_{n+1,\nu} \cap K_n \ne \varnothing.$ %$K_{n+1,\nu}$ is not a disc.
%In all other cases define the set  $K_{-1}:= K_n \cap K_{n+1,\nu}.$
Consider the set $K_{n+1,\nu} \cup {K}^{\bullet}_{-1}.$
It is a symmetric closed disc with finitely many  holes $D_1,\ldots D_{k_1}$ and contains (at most)  finitely many closed  symmetric discs $B_1,\ldots B_{k_2}$, which are connected components of ${K}^{\bullet}_{-1}.$ 
Recall that on $\partial B_j$ there are vertices of the tiling of $K_n$ to be avoided.
%Put $K:= K_{n+1,\nu}^{\bullet}$.

{\em (i) Extending the tiling  outwards.} Form a complete (symmetric) necklace of all $D_i$ and $B_j$ that
contain real points and let the symmetric $\delta_{n+1}$--tiling of the necklace be
given by Remark \ref{NeckTil} with vertices disjoint from vertices on tilings contained in the discs $B_j$. In the case of a trivial necklace,
tile the line segment.  Define the set $V$ as the set of the vertices
of this tiling of the necklace. Remove the union of the tiles from the set
$K_{n+1, \nu}\setminus \mathring{(K^{\bullet}_{-1})}$ and define $D$ to be  the connected
component of the new set in the upper half-plane. So $D$ is a closed
disc with a piecewise smooth boundary which contains all the remaining
discs $B_j$ and has all the remaining holes $D_i.$ As before, we collect them in the families
$\mathcal B$ and $\mathcal D.$
Proposition \ref{IndStep}  for $D$ and the set of vertices $V$ yields the desired tiling.

(ii) It remains to extend the tiling inwards, i.e. to tile the closure of the set  $(K_{-1}^{\bullet} \cap K_{n+1, \nu})\setminus {K_n}.$
%Choose a connected component $K_{n,\lambda}$ of $K_n$ and
Choose a connected component $B$ of $K_{-1}^{\bullet}.$
If $B$ is  a connected component of $K_n,$  it is already tiled, and there is nothing to do.
Hence, assume that $B$ is not a connected component of $K_n.$

If $B$ is in the upper (or lower) half-plane, then Case 1 applies.

Assume that $B$ intersects the real axis. As before, $K_{n}\cap B$
equals $B$ with finitely many holes, and each of them
contains a hole of $K_{n+1, \nu}$ (see Figure \ref{case2b}(a)). Since the holes of $K_{n}\cap B$
are disjoint, each one can be filled separately, so we may assume that
there is either a pair of symmetric holes (Figure \ref{case2b}(b) on the right) and then Case 1 applies or
there is only one (Figure \ref{case2b}(b), the light blue disc with centre at the origin) and intersects the real axis. Assume the latter and
denote the hole by $D'.$  We aim to extend the tiling to the set %$B' \cap K_{n+1,\lambda}$.
$D' \cap K_{n+1,\nu}$.

The set $(D' \cap K_{n+1, \nu}) \cup (K_n \cap D')^{\bullet}$ equals $D'$ with holes $D_1,\ldots , D_{k_1}$ and contains
filled components $B_1,\ldots, B_{k_2}$ of $K_{n}\cap D'.$ %different from $B.$
It may happen, that $k_2 = 0,$ but since $D'$ is a hole of $K_n,$ $k_1$ is necessarily strictly positive.
If $D_j$ for some $j \in \{1,\ldots,k_1\}$ is a hole on the real axis, then by reflection across small circle $S(a,r)\subset D_j,$ centred at the real axis, we transform the case to the first part of Case 2 (filling outwards).
 \begin{figure}[h]
 \centering
 $\begin{array}{ccccc}
  {\includegraphics[width=0.3\textwidth]{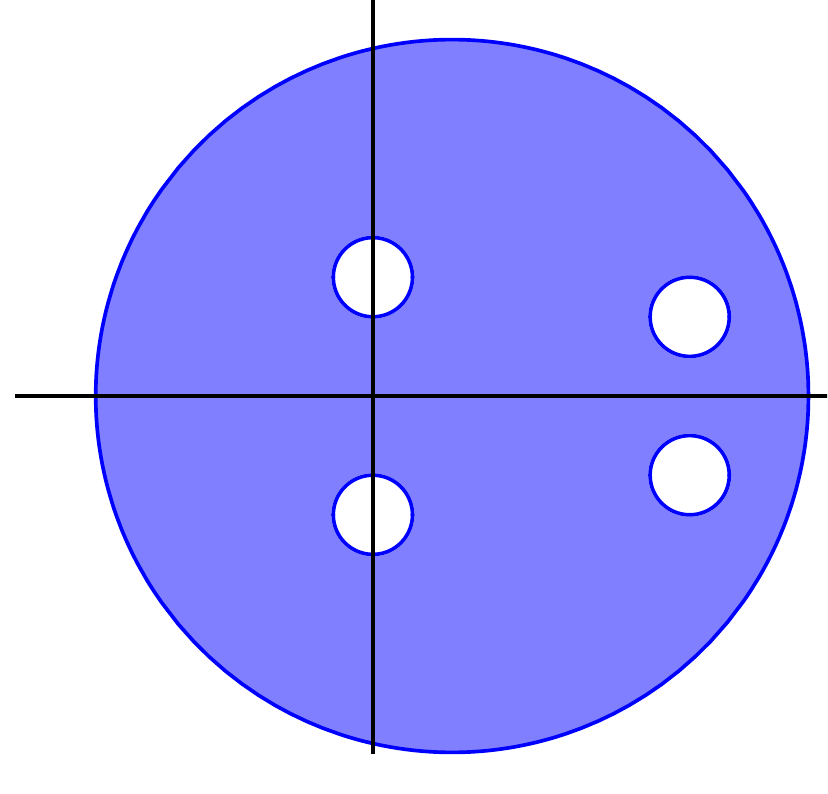}}&
  {\includegraphics[width=0.3\textwidth]{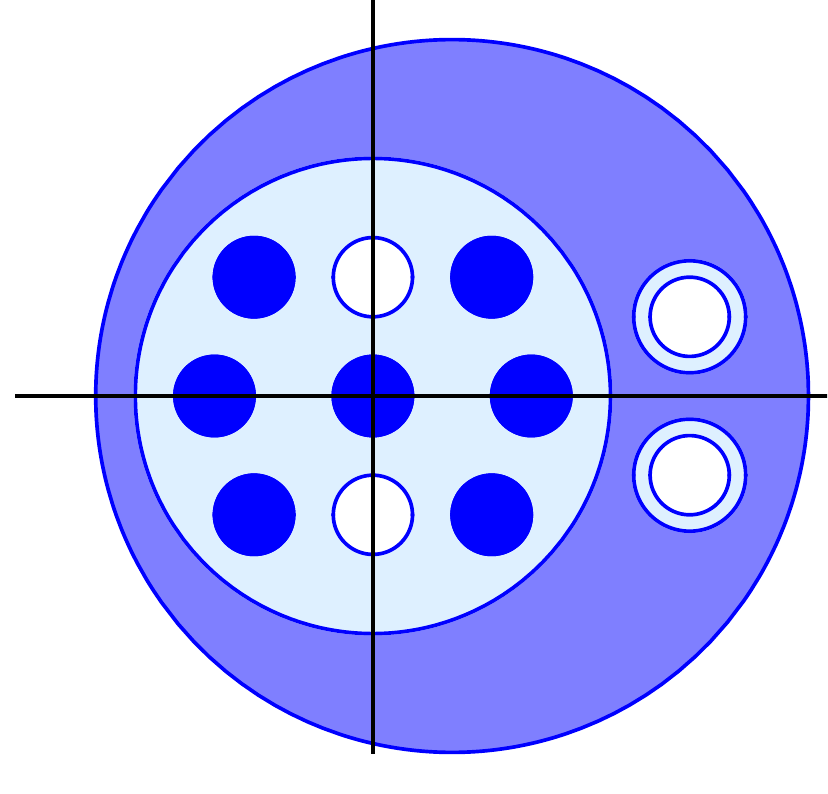}}&
  {\includegraphics[width=0.3\textwidth]{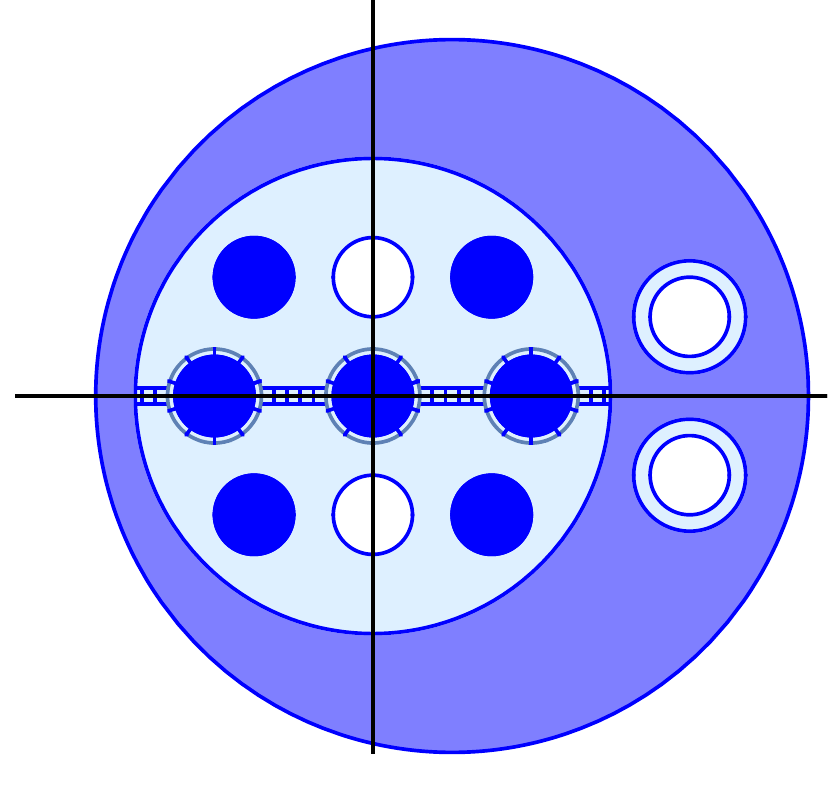}} \\
  (a)&(b)&(c)
 \end{array}$
  \caption{(a) the set $B \cap K_{n+1, \nu}$, (b) the set $K_{n}\cap B$  (blue), already tiled, with a hole $D$ (light blue centred at the origin) that intersects the real axis and with a pair of symmetric holes (light blue), filled components of $K_n$ (dark blue), domain to be tiled (light blue) and holes of $K_{n+1}$ (white), (c) the symmetric tiling of the complete symmetric necklace.}\label{case2b}
\end{figure}

If there is no hole on the real axis, then we form a complete symmetric (possibly trivial) necklace with the
filled components of $K_n$ contained in $D',$ which intersect the
real axis (Figure \ref{case2b}(c)) and extend the already constructed tiling with the tiling
of the necklace.  Remove the union of tiles of this tiling of the necklaces  from
the set $D',$  denote the closure of the connected component in the
upper half-plane by $K'$ and by $V$ the set of vertices of the tiling in $\partial K'.$ Then  it suffices to extend the tiling to $K'$ and  reflect it over the real axis to obtain the tiling of $D' \cap K_{n+1,\nu}.$

Recall that the outer boundary of $K'$ is already
contained in the tiling. Since there is at least one hole in $K'$,  the associated family of holes, $\mathcal{D}$, is not empty,
%we reflect the set $K'$ from any point from the hole's interior  and proceed with tiling outwards.
hence we can extend  the existing tiling with the one provided
by  Proposition \ref{IndStep} for the set $D:=K'$, the family $\mathcal D$ (and the family $\mathcal B$ of filled connected components of $K_n$, if it exists).
Observe that in this manner, we do get the extension of the tiling because the last tiles are added around the holes; therefore, it does not happen that the last tile intersects the union of the previously added tiles in a closed curve.

Notice that since $D'$ itself may contain the filled components of $K_n\cap D',$ which are not connected components of $K_n$,
we have to proceed inductively to extend the tiling to $ (K_n\cap D')^{\bullet} \cap K_{n+1,\nu}$.
Since there are only finitely many of
them, the process stops and defines a (finite) tiling of the compact
set $K_{n+1}.$  This completes the induction step.\qed

\section{Application to antisymmetric homology groups of symmetric planar domains}\label{applications}
%To add some text about $\mathcal{I} \mathb/b{Z}$ to be the appropriate replacement of $\Z$ when considering periods of exp.

{ As for the holomorphic case, one has to take into account  that the slice--preserving quaternionic exponential
$\exp$ has some periodicity, but this periodicity only applies when the function $\exp$  is restricted to a slice $\C_I$ and cannot be extended automatically;
to be more precise,  the function $f(z)=\exp(z+2\pi i)$   is periodic in $\C_i$ but  it is no longer
 slice--preserving (it is only $\C_i$ preserving), hence the periodicity of the  (extension) of the function $f$ to $\H$ is
not preserved.
On the other  hand the function $\exp(z)$ is $2\pi\mathcal{I}$--periodic in $\H\setminus\R$.

To give a consistent definition of fundamental domains of $\exp$ and its restrictions $\exp_I$ on each slice $\C_I$, it is convenient to extend the lattice $ i\Z$ in $\C_i$ and consider  the (image of) $\mathcal{I} \mathbb{Z}$ in $\H$ as its generalization.} Observe that, if $z\in\H\setminus\R$, then  \begin{equation}\mathcal{I}(z)=-\mathcal{I}(\overline{z}).\label{as}
\end{equation}

%%The following theorem is useful when tretating

%In what follows, we are going to investigate the periods of quaternionic exponential function
%in the perspective of looking for solutions of the multiplicative quaternionic Cousin problems.
%In a more general  setting these problems can be formulated in terms of vanishing of properly defined cohomology groups, which the authors will  present in a forthcoming paper.
%
%As usual, we proceed by first restricting the problem to slices. The property (\ref{as}) instructs as to consider antisymmetric functions on symmetric domains in $\C$.

 To investigate the periods of quaternionic exponential function
in looking for solutions to the multiplicative quaternionic Cousin problems,  we first restrict the problem to slices. The property (\ref{as}) instructs us to consider antisymmetric functions on symmetric domains in $\C$. Namely,  the periods of the function $\exp(z)$ are of the form $2 \pi n \mathcal{I}$ on $\bH\setminus\R$ and when restricted to a slice $\C_I$ we have values $2 \pi n I$ on the upper half-plane and values $-2 \pi n I$ on the lower half-plane.%\marginpar{Maybe we can announce a result from the other paper here?}

In a more general  setting, these problems can be formulated in terms of the vanishing of properly defined cohomology groups, which the authors will  present in a forthcoming paper. In particular,  we formulate an extension of Theorem \ref{AntisymmetricH2} in the quaternionic setting with the help of the function $\mathcal I$.

\begin{definition}[Antisymmetric complexes and cohomology groups]
Let $D\subset \mathbb{C}$ be an open symmetric set and $\mathcal{U} = \{U_{\lambda}\}_{ \lambda \in \Lambda}$ an open symmetric covering of $D$.
{\em An antisymmetric $m$-cochain} of $\mathcal{U}$  is any collection $\{(f_L,U_L), f_L : U_L \rightarrow \mathbb{Z}\}_{L \in \Lambda^m}$ with $f_L$ continuous, constant in $U_L^+$ and  satisfying the antisymmetric property:
\[f_L(z) = - f_L(\bar{z}).\]
Antisymmetric cochain complexes  and coboundary operators are defined as usual
 and corresponding cohomology groups denoted by $H^n_{a}(\mathcal{U},\mathbb{Z})$.
 The open symmetric coverings of $D$ form a directed set under refinement; therefore, we define $H^n_{a}(D,\mathbb{Z})$ to be a direct limit over symmetric open coverings of $D$.
\end{definition}

Notice that the definition of the antisymmetric $m$-cochain  implies that $f_L = 0$ whenever  $U_L \cap \R \ne \varnothing$.

\begin{theorem}[Vanishing of antisymmetric cohomology groups]\label{AntisymmetricH2} Let $D\subset \mathbb{C}$ be a symmetric  open set.
% and $\mathcal{U} = \{U_{\lambda}\}_{\lambda \in \Lambda}$ a   symmetric open covering of $D$.
Then
\[H^n_a(D, \Z) = 0 \mbox{ for all } n\geq 2.\]
\end{theorem}

%Having this at our disposal, which will appear in the forthcoming paper.
%\begin{theorem}
%Let $\Omega$ be an axially symmetric domain. Then % and $\mathcal{U}$ an axially symmetric covering of $\Omega$.
%%Consider a cochain complex with values in $\mathcal{I} \mathbb{Z}$. Then
%\[ H^n(\Omega,\mathcal{I}\mathbb{Z}) = 0 \mbox{ for } n \geq 2.\]
%\end{theorem}

\begin{proof}
Take $n = 2$. Without loss of generality we may assume, by Theorem \ref{BasicCovering}, that there exists a Cartan covering ${\mathcal{U}}$ of $\widetilde{D} \subset \bH,$ which defines an open symmetric covering $\mathcal{U}_I$ when restricted to $\mathbb{C}_I \cong \C$  and that  the  antisymmetric cocycle $C_{coc}$ is given by \[C_{coc} = \{(f_{klm}, U_{klm,I}), f_{klm}:U_{klm,I} \rightarrow \Z \}_{k,l,m \in \bN_0}.\]
  We would like to show  that $C_{coc}$ is a coboundary, i.e. there exists an antisymmetric
cochain $C=\{(f_{kl}, U_{kl,I}), f_{kl}:U_{kl,I} \rightarrow \Z \}_{k,l \in \bN_0}$  which is mapped to $C_{coc}$ by the coboundary operator.

By definition, if $U_{klm,I}$ consists of two symmetric components, $U_{klm,I} = U_{klm,I}^+ \cup
U_{klm,I}^-$, then $f_{klm}$ equals $n_{klm}$ on $U_{klm,I}^+$ and $-n_{klm}$ on $U_{klm,I}^-.$ If it has only one connected component, then
it intersects the real axis and so $n_{klm} = 0$.

 Define a new covering $\mathcal{V} = \{V_{\k}\}_{\k \in \bN_0}$,
 where the
 open sets are the connected components of the sets $U_{k,I}$,
 i.e. the
 sets $U_{k,I}^0:=U_{k,I}$ if connected and the sets $U_{k,I}^+$ and
 $U_{k,I}^-$ otherwise.

We  define the cocycle
$C_I:=\{(\nu_{\kappa\lambda\mu},V_{\kappa\lambda\mu})\}_{\kappa,\lambda,\mu
  \in \bN_0}$ in $D$ for the standard integer-valued cohomology groups in the following manner.
  If $V_{\kappa}$ is a connected component of  $U_{k,I},$ $V_{\lambda}$ a connected component of  $U_{l,I}$,
  $V_{\mu}$  a connected  component of  $U_{m,I}$, then we set
  \[\nu_{\kappa\lambda\mu} :=\left
\{\begin{array}{rl}
n_{klm}, & \mbox{ if }\quad V_{\kappa\lambda\mu} \subset \C_I^+,\\
-n_{klm}, & \mbox{ if }\quad V_{\kappa\lambda\mu} \subset \C_I^-,\\
0, & \mbox{ if }\quad V_{\kappa\lambda\mu} \cap  \R \ne \varnothing,
\end{array}
\right.\]
because if $V_{\kappa\lambda\mu} \cap \R \ne \varnothing,$ we have that $n_{klm}=0.$

Since we know that the standard cohomology group $H^2(D,\Z)$ is trivial, the cocycle $C_I$ is
a coboundary, hence there exists a cochain $C_1:=\{(\nu_{\kappa
  \lambda},V_{\kappa \lambda})\}_{ \kappa,\lambda \in \bN}$ so that
$\nu_{\kappa\lambda\mu} = \nu_{\kappa\lambda} + \nu_{\lambda\mu} +
  \nu_{\mu\kappa}$ in $V_{\kappa\lambda\mu}$.

% In order for $C$ to be   a cochain of  slice--preserving functions,
We are looking for a cochain with the following  antisymmetric properties; %in principle there is no reason that
if $V_{\kappa} = U_{k,I}^+,$ $V_{\lambda} = U_{l,I}^+$ and
$V_{\kappa_1} = U_{k,I}^-,$ $V_{\lambda_1} = U_{l,I}^-,$ we would also like  have
$\nu_{\kappa \lambda} = -\nu_{\kappa_1 \lambda_1}$, which is not generally the case. Using the notation
above, we define a new cochain by setting ${\mu}_{\kappa\lambda}:=
- \nu_{\kappa_1 \lambda_1}$ and ${\mu}_{\kappa_1\lambda_1}:= -
\nu_{\kappa \lambda}$. In the case $V_{\kappa} = U_{k,I}^0$
$V_{\lambda} = U_{l,I}^+,$   %$V_{\kappa_1} =V_{\kappa}= U_{k,I}^0,$
$V_{\lambda_1}= U_{l,I}^-$ we take $\kappa_1 = \kappa$ and use the same formula to define $\mu_{\kappa \lambda}$ as above.
In the case $V_{\kappa} = U_{k,I}^0,$ $V_{\lambda} = U_{l,I}^0,$ we define $\mu_{\kappa \lambda} :=-\nu_{\kappa \lambda} $.
%$V_{\kappa_1} = V_{\kappa}=U_{k,I}^0,$ $V_{\lambda_1} = V_{\lambda}=U_{l,I}^0$ the definition is analogous.
Set $N_{\kappa \lambda} = (\nu_{\kappa \lambda }
+ {\mu}_{\kappa \lambda})/2.$ Then, by construction, $N_{\kappa
  \lambda} = - N_{\kappa_1 \lambda_1}.$ In particular, if $V_{\kappa
  \lambda} \cap \R \ne \varnothing,$% $V_{\kappa \lambda}$ %is its own reflection and
  we get $N_{\kappa\lambda} = 0.$

Since the covering is symmetric, the cochain $C_2:=\{({\mu}_{\kappa \lambda}, V_{\kappa\lambda}), \kappa,
\lambda \in \bN\}$ is mapped to the cocycle $C_I$ and hence also   cochain $C_3:=\{(N_{\kappa\lambda}, V_{\kappa\lambda}), \kappa, \lambda \in \bN\}$ is mapped to $C_I$ by the coboundary operator.

To define the antisymmetric  cochain $C=\{(f_{kl}, U_{kl,I}), k,l \in \bN\}$ we proceed  as follows.
Let $V_{\kappa} \subset U_{k,I}$, $V_{\lambda} \subset U_{l,I}$. If  $U_{kl,I} \cap \R \ne \varnothing$, then it equals  $V_{\kappa \lambda};$
 we set $n_{kl} = N_{\kappa\lambda} (=0).$   If $U_{kl,I} \cap \R = \varnothing,$ then $U_{kl,I}=U_{kl,I}^+ \cup U_{kl,I}^-.$ If $U_{kl,I}^+ =  V_{\kappa \lambda}$ then define $f_{kl}:=N_{\kappa\lambda}$ on $U_{kl,I}^+$  and $f_{kl}:=-N_{\kappa\lambda}$ on $U_{kl,I}^-$. If $U_{kl,I}^- =  V_{\kappa \lambda}$ then define $f_{kl}:=N_{\kappa\lambda}$ on $U_{kl,I}^-$  and $f_{kl}:=-N_{\kappa\lambda}$ on $U_{kl,I}^+$.
 This implies that $C$ is indeed an antisymmetric cocycle that is  mapped to $C_{coc}$ by the coboundary operator thus making $C_{coc}$ a coboundary.

If $n \geq 3,$ the  groups $H^n_a(D,\Z)$ are trivial  because  there exist arbitrarily fine  Cartan coverings and they have order at most $3$, which means that no four distinct sets of such a covering intersect.
\end{proof}

\end{document}